\documentclass[preprint,english]{imsart}
\RequirePackage{amsthm,amsmath,amssymb,mathrsfs,mathtools,dsfont}
\RequirePackage[ruled,vlined]{algorithm2e}
\RequirePackage{algpseudocode}
\RequirePackage[numbers]{natbib}
\RequirePackage[colorlinks,citecolor=blue,urlcolor=blue]{hyperref}

\newtheorem{theorem}{Theorem}[section]

\newtheorem{lemma}[theorem]{Lemma}

\theoremstyle{definition}
\theoremstyle{remark}
\newtheorem{assumption}{Assumption}

\startlocaldefs
\newcommand{\preprintlb}{}

\newcommand{\eqa}{\begin{eqnarray}}
\newcommand{\ena}{\end{eqnarray}}
\newcommand{\eq}{\begin{equation}}
\newcommand{\en}{\end{equation}}
\newcommand{\eqs}{\begin{eqnarray*}}
\newcommand{\ens}{\end{eqnarray*}}

\newcommand{\mathd}{ \mathrm{d} }
\newcommand{\ind}{\mathds{1}}
\newcommand{\bigO}{\mathcal{O}}

\newcommand{\Rad}{\mathrm{Rad}}
\newcommand{\bv}[1]{\boldsymbol{#1}}
\newcommand{\cp}{\overset{\mathbb{P}}{\to}}

\def\ignore#1{}

\endlocaldefs

\begin{document}

\begin{frontmatter}
\title{Fast approximate simulation of finite long-range spin systems}
\runtitle{Fast simulation of spin systems}

\begin{aug}
\author{\fnms{Ross} \snm{McVinish}\thanksref{t2}\ead[label=e2]{r.mcvinish@uq.edu.au}}
\and
\author{\fnms{Liam} \snm{Hodgkinson}\thanksref{t1}\thanksref{t2}\ead[label=e1]{liam.hodgkinson@uqconnect.edu.au}}

\thankstext{t1}{Supported by an Australian Postgraduate Award.}
\thankstext{t2}{All authors are supported in part by the Australian Research Council (Discovery Grant DP150101459 and the ARC Centre of Excellence for 
Mathematical and Statistical Frontiers, CE140100049).}

\runauthor{R. McVinish, L. Hodgkinson}

\affiliation{University of Queensland}

\address{
School of Mathematics and Physics\\
University of Queensland\\
St. Lucia, Brisbane \\
Queensland, 4072\\
Australia \\
\printead{e2}\\
\phantom{E-mail:\ }
\printead*{e1}
}

\end{aug}

\begin{abstract}
Tau leaping is a popular method for performing fast approximate simulation of certain continuous time Markov chain models typically found in chemistry and biochemistry. This method is known to perform well when the transition rates satisfy some form of scaling behaviour. In a similar spirit to tau leaping, we propose a new method for approximate simulation of spin systems which approximates the evolution of spin at each site between sampling epochs as an independent two-state Markov chain. When combined with fast summation methods, our method offers considerable improvement in speed over the standard Doob-Gillespie algorithm. We provide a detailed analysis of the error incurred for both the number of sites incorrectly labelled and for linear functions of the state. 
\end{abstract}



\end{frontmatter}

\maketitle

\section{Introduction}
\setcounter{equation}{0}

Finite spin systems are continuous-time Markov \preprintlb chains on $ \{0,1\}^{S} $, where $ S $ is a finite collection of sites, and the value at only one site changes in each transition. Special cases include the contact process, voter model and Ising model.  While great advances have been made in the analysis of these systems in relation to qualitative issues such as phase transitions, rates of convergence, and central limit theorems \cite{Liggett:1999,Liggett:2005,Liggett:2010}, to address more quantitative questions it is often necessary to employ simulation.  As a finite-state continuous-time Markov chain, sample paths can be constructed from the jump chain and holding time distribution based on the work of Doob \cite{Doob:1942,Doob:1945}. 
In the chemistry and biochemistry literature this algorithm is often referred to as the \emph{Gillespie algorithm} following \cite{Gillespie:1976,Gillespie:1977}; in the sequel, shall be referred to as the \emph{Doob--Gillespie algorithm}.

When the number of sites is large, exact simulation over a desired time interval can be infeasible due to prohibitively small expected times between transitions. This issue is not particular to spin systems and arises more generally in the simulation of continuous time Markov chains modelling large populations. Gillespie \cite{Gillespie:2001} proposed the \emph{tau-leaping method} as a means of generating an approximate sample path that avoids this problem. Tau-leaping essentially treats the transition rates as constant over small time intervals. Anderson et al. \cite{AGK:2011} provide a detailed analysis of resulting error in the sample paths. In principle, tau-leaping could be applied to spin systems, however they lack the scaling required for tau-leaping to provide a reasonable speed-accuracy trade off.

In this paper we propose a new algorithm for the approximate simulation of long-range finite spin systems. The basic idea here is to briefly decouple sites over small time intervals so that the spin system is treated as a collection of independent two-state Markov chains. We provide a detailed analysis of the error in terms of the number of sites incorrectly labelled and accuracy of certain linear functionals of the state. Our analysis is inspired by \cite{AGK:2011}, though we keep our analysis entirely non-asymptotic.

\subsection{The basic model} 

We formulate the class of \emph{finite spin systems} as continuous-time Markov processes on the state space $\{0,1\}^n$ for some positive integer $n$ representing the total number of sites, where simultaneous transitions at multiple distinct sites occurs with zero probability. Any finite spin system $ (\eta(t), t \geq 0) $ can be represented as a Markov jump process in the usual transition notation:
\[
\begin{aligned}
\eta_i:\quad 0\to 1 &\quad \mbox{at rate }q_i^{+}(\eta)\\
1\to 0 &\quad \mbox{at rate }q_i^{-}(\eta)
\end{aligned} \qquad \mbox{for }i=1,\dots,n.
\]
Here, the functions $q_i^+,q_i^- : \{0,1\}^n \to [0,\infty)$ represent the rate at which sites flip to the positive state (1), and the zero state (0), respectively. The function $q(\eta,i) = \eta_i q_i^{-}(\eta) + (1-\eta_i)q_i^{+}(\eta)$ is often called the \emph{rate function} of the process \cite{Liggett:2005}. Let $N_i^+$ and $N_i^-$ for each $i=1,\dots,n$ be independent unit-rate Poisson processes. Representing the finite spin system $\eta$ as a random time change of Poisson processes \cite[\S6]{EthierKurtz:2005}, we find that for any $t \geq 0$,
\begin{multline}
\label{model:eq4}
\eta_i(t) = \eta_i(0) + N_i^{+}\left(\int_0^t (1-\eta_i(s))q_i^{+}(\eta(s)) \mathd s\right) \\- N_i^{-}\left(\int_0^t \eta_i(s) q_i^{-}(\eta(s)) \mathd s\right),\quad i=1,\dots,n.
\end{multline}

The focus in this paper is on a predominant subclass of practical spin systems whose $ q_{i}^{+} $ and $ q_{i}^{-} $ functions are of the general form
\begin{equation}
f_{i}(v_i),\quad\mbox{where}\quad v_i = \sum_{j=1}^{n}s_{ij}x_{j}, \label{model:eq7}
\end{equation}
for some functions $f_{i}:\mathbb{R}\to\mathbb{R}_{+}$ and constants $s_{ij}\in\mathbb{R}$ such that the sums $ \sum_{j=1}^{n} s_{ij} $ are uniformly bounded independently of $n$. As a model of mean-field type, $\{v_i\}_{i=1}^n$ are commonly referred to as \emph{potentials}. Transition rate functions of this form are common in applications such as Hanski's metapopulation model \cite{Hanski:1994,AlonsoMcKane:2002}, spatial SIS epidemic model \cite{BTK:2015,HamadaTakasu:2019}, the voter model \cite[\S5]{Liggett:2005} and the Ising model with Kac potentials \cite{DeMasi:2001}. They are often considered for their capacity to interact with a very large number of other sites (indeed, the entire system!). The form of (\ref{model:eq7}) suggests that it will be important to understand the evolution of linear functionals of the state. Under certain conditions, the spin system is well approximated by a deterministic system $\rho(t) = (\rho_i(t))_{i=1}^n$ in the sense that $ \sup_{\phi \in \Phi} \left| \sum_{i=1}^n \phi_{i} (\eta_{i}(t) - \rho_{i}(t)) \right| $ is small for sufficiently regular $ \Phi \subset \mathbb{R}^{n} $ \cite{BMP:15}. In particular, this deterministic system is given by the solution to
\begin{equation}
\rho_i(t) = \rho_i(0) + \int_0^t (1-\rho_i(s))q_i^+(\rho(s)) -\rho_i(s)q_i^{-}(\rho(s)) \mathd s,\quad \mbox{for all }t \geq 0, \label{ODE:eq1}
\end{equation}
assuming that $q_{i}^+ $ and $ q_{i}^-$ can be smoothly extended to $[0,1]^n$ (see Appendix \ref{sec:IndepSite}). 

\subsection{Approximate simulation} As noted earlier, finite spin systems can in principle be simulated exactly using the Doob-Gillespie algorithm, though this may be computationally infeasible. Assuming that for each $ \eta \in \{0,1\}^{n} $, $\sum_{i=1}^{n} q(\eta,i) = \mathcal{O}(n) $, the expected number of transitions in a sample path on $ [0,T] $ is then $ \mathcal{O}(nT)  $. If each of the transition rates can be updated after a transition in constant time, as is the case for transition rate functions above by storing and updating the potentials, then the computational cost of simulating this sample path using the Doob-Gillespie algorithm is $ \mathcal{O}(n^{2} T) $. 

Basic tau leaping is ill-suited to simulating spin systems. Fixing some step size $\delta > 0$, let $t_k \coloneqq k \delta$ denote the lattice of points over which the simulation is carried out. For notational convenience, we define $\chi$ as the simple function
\[
\chi(s) \coloneqq \sum_{k=0}^{\infty} t_k \ind_{s \in [t_k, t_{k+1})}.
\]
that maps each time point $t$ to the last point on the lattice. The process of Euler tau-leaping translates to holding the transition rates constant over each interval $[t_k,t_{k+1})$. With this in mind, the Euler tau-leaping approximation $\hat{\eta}^{\delta}$ to (\ref{model:eq4}) satisfies
\begin{multline}
\hat{\eta}^{\delta}_i(t) = \eta_i(0) + N_i^+\left(\int_0^t (1-\hat{\eta}_i^{\delta}\circ \chi(s)) \, q_i^+(\hat{\eta}^{\delta} \circ \chi(s)) \mathd s \right) \\
- N_i^-\left(\int_0^t \hat{\eta}_i^{\delta}\circ \chi(s) \, q_i^{-}(\hat{\eta}^{\delta} \circ \chi(s))\mathd s \right),\qquad i=1,\dots,n.
\end{multline}
At the lattice points, $ \hat{\eta}^{\delta} $ can be evaluated recursively via
\[
\hat{\eta}^{\delta}_{i}(t_{k+1}) = \hat{\eta}^{\delta}_{i}(t_{k}) + \Lambda_{i,k}^{+} - \Lambda_{i,k}^{-},
\]
where $ \Lambda_{i,k}^{+} $ and $ \Lambda_{i,k}^{-} $ are independent Poisson random variables with means $ \delta\hat{\eta}^{\delta}_{i} q^{+}_{i}(\hat{\eta}^{\delta} (t_{k})) $ and $\delta (1-\hat{\eta}^{\delta}_{i}) q^{-}_{i}(\hat{\eta}^{\delta} (t_{k})) $, respectively. For the path of $ \hat{\eta}^{\delta} $ to remain in $ \{0,1\}^{n} $ it is necessary that the Poisson random variables $ \Lambda_{i,k}^{+} $ and $ \Lambda_{i,k}^{-} $ are either 0 or 1 for all $ i $ and $ k $. For this to occur with high probability it is necessary that the step size is $ \mathcal{O}(n^{-1}) $. The need to take such a small step size removes any benefit to performing tau-leaping. 

Motivated by the {\em propagation of chaos} results that often hold for this type of system \citep{Leonard:90,Sznitman:91,BarbourLuczak:15}, our strategy is to decouple the sites over each time interval $[t_k,t_{k+1})$ so that the process becomes a system of $n$ independent two-state Markov chains whose transition probabilities are known explicitly. This process never leaves the set $ \{0,1\}^{n}$ so avoids the issue described above. However, the quality of the approximation will depend on how well $ q_{i}^{+}(\eta(t)) $ and $ q_{i}^{-}(\eta(t)) $ can be approximated over $[t_{k},t_{k+1}] $ by suitably chosen constants. The simplest way of incorporating this idea is via a forward Euler scheme, approximating  $ q_{i}^{+}(\eta(t)) $ and $ q_{i}^{-}(\eta(t)) $  by their value at the most recent point on the lattice. With this in mind, the Euler approximation $\eta^{\delta}$ to (\ref{model:eq4}) satisfies
\begin{multline}
\label{model:eq5}
\eta^{\delta}_i(t) = \eta_i(0) + N_i^+\left(\int_0^t (1-\eta_i^{\delta}(s))q_i^+(\eta_i^{\delta} \circ \chi(s)) \mathd s \right) \\
- N_i^-\left(\int_0^t \eta_i^{\delta}(s)q_i^{-}(\eta_i^{\delta} \circ \chi(s))\mathd s \right),\qquad i=1,\dots,n,
\end{multline} 
The resulting two-state Markov chains on $ [t_{k}, t_{k+1}] $ can be simulated using the Doob-Gillespie algorithm . Even more convenient is the fact that the transition probabilities of two-state Markov jump processes are known explicitly, and so it is possible to simulate the process at each time point $t_k$ as a discrete-time Markov chain. Indeed, by letting $Q_i(x) = q_i^+(x) + q_i^-(x)$
for each $i=1,\dots,n$, assuming $Q_i(x) \neq 0$, $\mathbb{P}(\eta_i^{\delta}(t_{k+1}) = 1\vert \eta^{\delta}(t_k)) = P_i^{\delta}(\eta^{\delta}(t_k))$, where
\begin{equation}
\label{eq:EulerTransitions}
P_i^{\delta}(\eta^{\delta}(t_k)) \coloneqq
\begin{cases}
\frac{q_i^+(\eta^{\delta}(t_k))}{Q_i(\eta^{\delta}(t_k))}(1 - e^{-\delta Q_i(\eta^{\delta}(t_k))}) & \mbox{ if } \eta_i^{\delta}(t_k) = 0 \\
1 - \frac{q_i^-(\eta^{\delta}(t_k))}{Q_i(\eta^{\delta}(t_k))}(1 - e^{-\delta Q_i(\eta^{\delta}(t_k))}) & \mbox{ if } \eta_i^{\delta}(t_k) = 1.
\end{cases}
\end{equation}
Pseudocode for simulating the Euler approximation on the lattice $\{k\delta\}_{k=0}^{N}$ is provided in Algorithm \ref{alg:EulerApprox}.

\begin{algorithm}[htbp]
	\SetAlgoLined
	\SetKwInOut{Input}{Input}\SetKwInOut{Output}{Output}
	\Input{Number of sites $n$, transition rates $q_i^+$, $q_i^-$ for $i=1,\dots,n$, initial value $\eta^{\delta}(0) \in \{0,1\}^n$, time step $\delta > 0$, and number of time steps $N$ }
    \Output{System state at each time point $\{\eta^{\delta}(\delta),\eta^{\delta}(2\delta),\dots,\eta^{\delta}(N\delta)\}$} ~ \\
	\For{$k=0,\ldots,N-1$}{
	\For{$i=1,\dots,n$}{
	Sample $\eta_i^{\delta}((k+1)\delta)$ as a Bernoulli random variable with success probability $P_i^{\delta}(\eta^{\delta}(k\delta))$. 
	}
	}
	\caption{Simulating the Euler approximation (\ref{model:eq5}) for spin systems}
	\label{alg:EulerApprox}
\end{algorithm}

To improve upon the Euler method, a midpoint approximation to $ q_{i}^{+}(\eta(s)) $ and $ q_{i}^{-}(\eta(s)) $ will also be considered. Analogous to the deterministic approximation (\ref{ODE:eq1}), this requires extending $ q_{i}^{+} $ and $ q_{i}^{-} $ smoothly to $ [0,1]^{n} $. This can be done in an arbitrary manner, but the choice of extension will play a role in the quality of the approximation, and consequently, in the error analysis to follow. Letting $p_i^{\delta}(z) \coloneqq z_i + \frac12 \delta q_i(z)$ for $z \in [0,1]^n$ and each $i=1,\dots,n$, and assuming that $\delta$ is sufficiently small so that $p_i^{\delta}$ maps $[0,1]^n$ into itself, the midpoint approximation $\mathring{\eta}^{\delta}$ satisfies
\begin{multline}
\label{model:eq6}
\mathring{\eta}^{\delta}_i(t) = \eta_i(0) + N_i^+\left(\int_0^t (1-\mathring{\eta}_i^{\delta}(s))q_i^+(p^{\delta} \circ \mathring{\eta}^{\delta} \circ \chi(s)) \mathd s \right) \\
- N_i^-\left(\int_0^t \mathring{\eta}_i^{\delta}(s)q_i^{-}(p^{\delta} \circ \mathring{\eta}^{\delta} \circ \chi(s))\mathd s \right),\qquad i=1,\dots,n.
\end{multline}
This time, the transition probabilities are given by
\begin{equation}
\label{eq:MidpointTransitions}
\mathbb{P}(\mathring{\eta}_i^{\delta}(t_{k+1}) = 1 \vert \mathring{\eta}_i^{\delta}(t_k)) = P_i^{\delta}(p^{\delta} \circ \mathring{\eta}^{\delta}(t_k)).
\end{equation}
Pseudocode for simulating the midpoint approximation on the lattice $\{k\delta\}_{k=0}^{N}$ is provided in Algorithm \ref{alg:MidpointApprox}.

\begin{algorithm}[htbp]
	\SetAlgoLined
	\SetKwInOut{Input}{Input}\SetKwInOut{Output}{Output}
	\Input{Number of sites $n$, transition rates $q_i^+$, $q_i^-$ for $i=1,\dots,n$, initial value $\mathring{\eta}^{\delta}(0) \in \{0,1\}^n$, time step $\delta > 0$, and number of time steps $N$ }
    \Output{System state at each time point $\{\mathring{\eta}^{\delta}(\delta),\mathring{\eta}^{\delta}(2\delta),\dots,\mathring{\eta}^{\delta}(N\delta)\}$} ~ \\
	\For{$k=0,\ldots,N-1$}{
	\For{$i=1,\dots,n$}{
	Compute $\tilde{\eta}_i^{\delta}(k\delta) = p_i^{\delta}(\mathring{\eta}^{\delta}(k\delta))$.
	}
	\For{$i=1,\dots,n$}{
	Sample $\mathring{\eta}^{\delta}((k+1)\delta)$ as a Bernoulli random variable with success probability $P_i^{\delta}(\tilde{\eta}^{\delta}(k\delta))$. 
	}
	}
	\caption{Simulating the midpoint approximation (\ref{model:eq6}) for spin systems}
	\label{alg:MidpointApprox}
\end{algorithm}

Practically speaking, a single step of Algorithm \ref{alg:MidpointApprox} is not much more challenging to execute than Algorithm \ref{alg:EulerApprox}, only doubling the computation involved. However, as we shall see, the midpoint approximation can be substantially more accurate than the Euler approximation allowing a much larger step size to be taken.

\subsection{Fast summation methods}

Algorithms \ref{alg:EulerApprox} and \ref{alg:MidpointApprox} will only be useful if they can be implemented with significantly less computational cost than the Doob-Gillepie algorithm applied to the original process (\ref{model:eq4}). When the functions $ q_{i}^{+} $ and $ q_{i}^{-} $ have the form (\ref{model:eq7}), naive evaluation of the $ 2n $ transition rates requires  $\bigO(n^2)$ operations. In Algorithms \ref{alg:EulerApprox} and \ref{alg:MidpointApprox} they must be recomputed at every time-step, implying an $\bigO(n^2 T/\delta)$ computational burden. On the other hand the computation burden of the Doob-Gillespie algorithm is $ \bigO(n^{2} T) $. Therefore, naively implemented, these algorithms are more costly to implement than the Doob-Gillespie algorithm. For Algorithms \ref{alg:EulerApprox} and \ref{alg:MidpointApprox} to offer computational savings we need methods which can evaluate $ n $ sums of the form (\ref{model:eq7})  in  sub-quadratic time. 

In the special case where the sites are indexed on a $ d$-dimensional regular lattice and $ s_{ij} = k(i-j) $ for some kernel function $ k $, the sums in (\ref{model:eq7}) become convolutions which may be rapidly computed exactly in $ \bigO(n\log n) $ time using the fast Fourier transform. This approach is simple and well-known, yet remains effective and highly recommended whenever applicable, thanks to the accessibility of highly efficient implementations of the fast Fourier transform \cite{frigo1998fftw}. In fact, our numerical experiments will focus predominantly on this particular method.

The state-of-the-art methods in this arena are the class of tree methods and their siblings. When the number of spatial dimensions $ d $ is not too large, these comprise perhaps the fastest summation methods to date \cite{Gray2001}. The simpler, single-tree case involves the construction of a tree which groups the $ n $ sites together in $ \bigO(\log n) $ levels, in $ \bigO (n\log n) $ time. Usually this is performed spatially with respect to some metric (usually Euclidean). By utilising some form of approximation (either spatial averages or multipole expansions), information necessary to estimate potentials can be contained in each group at each level of the tree. For each source site $z_i$, the corresponding potential $v_i$ can be obtained by passing through each level of the tree. The number of steps required to estimate the potential of a single node becomes $\bigO(\log n)$, for a total $\bigO(n \log n)$ computation time. Algorithms of this form include the celebrated Barnes-Hut algorithm \cite{Barnes1986} commonly used for simulations of the $n$-body problem.

A more sophisticated, and often much more efficient, approach, is to group source sites according
to a tree as well. Such methods are referred to as dual–tree methods, and have linear total computation complexity, for a fixed tolerance $\epsilon$ \cite{Gray2001}. The most noteworthy dual–tree method is the fast multipole method of
Greengard and Rokhlin \cite{Greengard1987} (including its application to the Gaussian kernel, called the
fast Gauss transform \cite{Greengard1991,Yang2003}). Provided the bandwidth of the kernel is not too small, computation
of the potentials can be achieved in $\bigO(n)$ time. On the other hand, if the bandwidth is too small, the
fast multipole method can prove even less efficient than the naive approach.

\subsection{Paper outline} The remainder of the paper is organised as follows: first, in \S\ref{sec:StrongError}, we develop strong error bounds for the approximations (\ref{model:eq5}) and (\ref{model:eq6}), controlling the expected number of sites with incorrect spin in terms of the step size and norms on the derivatives of $ q $. In \S\ref{sec:ExactError}, we study the renormalised differences of linear combinations of the state vectors for the Euler method. Unlike the approach seen in \cite{AGK:2011}, ours will be purely non-asymptotic, deriving an explicit error bound on the rate of convergence as well. Together with the strong error bounds, this analysis shows that the mid-point method is substantially more accurate than the Euler method when the system is not in equilibrium. Finally, in \S\ref{sec:Numerics}, we empirically demonstrate the accuracy of the approximation and the reduced computational cost of the proposed simulation method. All proofs have been relegated to Appendix \ref{sec:Proofs}.

\section{Strong error analysis}
\label{sec:StrongError}
Our analysis begins with the development of strong error bounds for the Euler and
midpoint methods. These arguments adhere fairly closely to those of 
\cite{AGK:2011}, but must take greater care to keep track of terms
depending on $n$, especially those involving derivatives of $q^+$ and $q^-$.
For this purpose, it is necessary to introduce some notation.
For a function $f : [0,1]^{n} \rightarrow \mathbb{R}^{m}$, let $ \|f\|_{\infty} = \max_{i} \sup_{x} |f_{i}(x)|$. 
Define
\begin{align*}
\|Df\|_{1} & \coloneqq \max_{j} \sum_{i} \|\partial_{j} f_{i} \|_{\infty}, & \|D^{\ast}f\|_{1} & \coloneqq \max_{j} \sum_{i\neq j} \|\partial_{j} f_{i} \|_{\infty}, \\
\|Df\|_{\infty} &\coloneqq \max_{i} \sum_{j} \|\partial_{j} f_{i} \|_{\infty}, & \|D^{\ast}f\|_{\infty} & \coloneqq \max_{i} \sum_{j\neq i} \|\partial_{j} f_{i} \|_{\infty}.
\end{align*}
Since $ q_{i}(x) = q_{i}^+(x) $ if $ x_{i} = 0 $ and $ q_{i}(x) = q_{i}^-(x) $ if $ x_{i} = 1$, it follows that $ \|q^+\|_{\infty} \leq \|q\|_{\infty} $ and $ \|q^-\|_{\infty} \leq \|q\|_{\infty} $. 
Furthermore, by construction, $ \partial_{i} q^+_{i}(x) = \partial_{i} q^-_{i} (x) = 0 $ and, for $ j\neq i $, $ \partial_{j} q_{i}(x) = (1-x_{i}) \partial_{j} q^+_{i}(x) - x_{i} \partial_{j} q^-_{i}(x) $, Therefore, if $p = 1$ or $p = \infty$, $\|Dq^+\|_{p} \leq \|D^{\ast} q\|_{p}$ and $\|Dq^-\|_{p} \leq \|D^{\ast} q\|_{p}$. 


With this in tow, the strong error analysis for the Euler method is stated in Theorem \ref{thm:EulerRate}. Note that the only assumption required here is that $q \in \mathcal{C}^1([0,1]^n,\mathbb{R}_+^n)$. 
\begin{theorem}
\label{thm:EulerRate}
For some step-size $\delta > 0$, let $\eta(t)$ and $\eta^{\delta}(t)$ satisfy (\ref{model:eq4}) and the Euler approximation (\ref{model:eq5}) respectively.
Then for any time $T >  0$,
\[
\sup_{t \in [0,T]} \sum_{i=1}^n \mathbb{E}|\eta_i(t) - \eta_i^{\delta}(t)| \leq 4 n \delta T \|q\|_{\infty} \|D^{\ast}q\|_1 e^{2T(\|q\|_{\infty}+\|D^{\ast}q\|_1)}.
\]
\end{theorem}

To elucidate the order of approximation of the strong error bound in 
Theorem \ref{thm:EulerRate}, we return to the aforementioned `mean-field' case. Here,
both $\|q\|_{\infty}$ and $\|D^{\ast}q\|_1$ are $\mathcal{O}(1)$, and so
the expected number of discrepancies between $\eta(t)$ and the Euler approximation
$\eta^{\delta}(t)$ is $\mathcal{O}(n\delta)$. Recalling that the corresponding error for the
deterministic approximation is $\mathcal{O}(n^{1/2})$, the Euler
approximation can be guaranteed to be competitive only if $\delta = \mathcal{O}(n^{-1/2})$.

We now proceed on to obtain a strong error bound for the midpoint method. The situation here is more complex, now involving second-order
derivatives of $q$, so requiring $q \in \mathcal{C}^2([0,1]^n,\mathbb{R}_+^n)$. Once again, we define two variables that control this higher-order regularity:
\[
\gamma_n = \sum_{\substack{i,j=1 \\ i \neq j}}^{n} \|\partial^{2}_{j} q_{i} \|_{\infty},\qquad \mbox{and}\qquad\Gamma_n = \max_{i=1,\dots,n} \sum_{\substack{j,k=1\\ j,k\neq i}}^n \|\partial_{j} \partial_{k} q_{i}\|_{\infty} .
\]
For a function $f:[0,1]^n \to \mathbb{R}^m$, we will also require
\[
\|Df\|_{2,1} = \sum_{i=1}^n\left(\sum_{j=1}^n\|\partial_j f_i\|_{\infty}^2\right)^{1/2},\ 
\|D^{\ast}f\|_{2,1} = \sum_{i=1}^m\left(\sum_{j=1,j\neq i}^n \|\partial_j f_i\|_{\infty}^2\right)^{1/2}.
\]
The strong error bound for the midpoint method is presented in Theorem \ref{thm:MidPointRate}.
\begin{theorem}
\label{thm:MidPointRate}
Let $\eta(t)$ and $\mathring{\eta}^{\delta}(t)$ satisfy (\ref{model:eq4}) and the midpoint
approximation (\ref{model:eq6}) respectively. Then for any step-size $\delta > 0$
and any time $T > 0$,
\[
\sup_{t \in [0,T]} \sum_{i=1}^n \mathbb{E}|\eta_i(t) - \mathring{\eta}^{\delta}_i(t)| \leq 10 \alpha(n,\delta) (T+1) e^{2T(\|q\|_{\infty} + \|D^{\ast} q\|_1)}.
\]
where 
\begin{multline*}
\alpha(n,\delta) =  n\delta^{2} \|q\|_{\infty}(1+ \|q\|_{\infty}) (1+\|D^{\ast}q\|_1)\left( \Gamma_{n} +  \|D^{\ast} q\|_1 \right) \\ + \delta \|q\|_{\infty} \gamma_{n} +\delta^{1/2} \|q\|_{\infty}^{1/2} \|D^{\ast} q\|_{2,1}.
\end{multline*}
\end{theorem}

Returning to the mean-field case, both $\gamma_n$ and $\Gamma_n$ are $\mathcal{O}(1)$,
while $\|D^{\ast}q\|_{2,1} $ is  $ \mathcal{O}(n^{1/2})$, implying the error bound
in Theorem \ref{thm:MidPointRate} is of order $\mathcal{O}(n\delta^2 + (n\delta)^{1/2})$.
If $\delta \geq n^{-1/3}$, the $\mathcal{O}(n\delta^2)$ term dominates as
$n\to\infty$, implying that the midpoint method is competitive with the
deterministic approximation if $\delta = \mathcal{O}(n^{-1/4})$. 
Ignoring constants, for large $n$, this suggests that the midpoint method 
has smaller error than the Euler method for the viable case $\delta \geq n^{-1}$. 

\section{Exact error asymptotics}
\label{sec:ExactError}
Having now established strong error bounds for the two approximation methods, we shall formally demonstrate that the midpoint method should outperform the Euler method in the $n \to \infty$ regime for viable step sizes. To accomplish
this, we will prove that the $\mathcal{O}(n\delta)$ error complexity is \emph{exact} by characterising the limiting behaviour of weighted errors
\[
(n\delta)^{-1}\sum_{i=1}^n \phi_i (\eta_i(t)-\eta^{\delta}_i(t)), \quad\phi_i \in \mathbb{R},
\]
between the finite spin system $\eta(t)$ and the Euler approximation $\eta^{\delta}(t)$,
as $n\delta\to\infty$. More specifically, these errors will be shown in Theorem \ref{thm:EulerExact} to be well-approximated
by $n^{-1}\sum_{i=1}^n \phi_i \mathcal{E}_i(t)$, where $\mathcal{E}(t)$
is the solution to the system of ordinary differential
equations:
\begin{subequations}
\begin{align}
\frac{\mathd}{\mathd t} \mathcal{E}_i(t) &= Dq(\rho(t))^\top\mathcal{E}(t)
+ \frac12 D^{\ast}q(\rho(t))^\top q(\rho(t)) \label{eq:mathcalE}\\
\mathcal{E}_i(0) &= 0.
\end{align}
\end{subequations}
Using Gronwall's inequality, the growth of $\mathcal{E}(t)$ can be readily controlled. Since for each $i~=~1,\dots,n$ and any $t \geq 0$, \[
\mathcal{E}_i(t) \leq \int_0^t \|Dq\|_{1} \|\mathcal{E}(s)\|_{\infty} \mathd s
+ \frac12 \|D^{\ast}q\|_{1} \|q\|_{\infty} t,
\]
for any $T > 0$, 
\begin{equation}
\label{eq:EGrowth}
\sup_{t\in[0,T]}\|\mathcal{E}(t)\|_{\infty} \leq \tfrac12 \|q\|_{\infty}\|D^{\ast}q\|_{1} T e^{T\|Dq\|_1},
\end{equation}
which grows comparably to the error bound in Theorem \ref{thm:EulerRate}. 


%

\subsection{Assumptions}
Although the strong error bounds have been obtained in virtually complete generality, unfortunately, the same could not be achieved for the exact error asymptotics. Indeed, some regularity between the functions $q_i$, $i=1,\dots,n$, needs to be enforced. This can be accomplished by parameterisation, assuming that $q_i(\bv x) = q(\bv x; \bv z_i)$ for each $i=1,\dots,n$. Now, one need only impose appropriate assumptions on $q$ and $\bv z_i$. 

The idea here will be to take advantage of the small metric entropy and closure under pointwise multiplication of the following set of analytic functions. Let $\Omega \subset \mathbb{C}^d$ be a compact parameter space and let $\Phi$ denote the set of functions $\phi:\Omega\to\mathbb{R}$ analytic on $\Omega$, each with an analytic continuation to a function $\phi^{\ast}:\Omega'\to\mathbb{C}$ where $\Omega \subset \Omega' \subset \mathbb{C}^d$ satisfying $|\phi^{\ast}(z)| \leq 1$ for
$z \in \Omega'$. For each $\epsilon > 0$, we assign an
$\epsilon$-internal covering $\Phi_\epsilon \subset \Phi$ of minimal cardinality under the uniform metric. From \cite[Theorem 9.2]{vitushkin}, there exists a constant $C_{\Omega,\Omega',d}$ depending only on $\Omega,\Omega'$ and their dimension $d$ such that for any $0 < \epsilon < 1$,
$\log|\Phi_{\epsilon}| \leq C_{\Omega,\Omega',d} |\log \epsilon|^{d+1}$.
Also note that, for any
chosen parameters $z_1,\dots,z_n \in \Omega$, we can bound
the Rademacher complexity of the set of vectors $\Phi(\bv z) = \{(\phi(z_i))_{i=1}^n\,:\,\phi\in\Phi\}$. Indeed, letting $\sigma_1,\dots,\sigma_n$ denote independent Rademacher random variables, using similar arguments to \cite[Theorem 3.2]{Devroye:2001aa},
\begin{equation}
\label{eq:RadBound}
\Rad(\Phi(\bv z)) \coloneqq \mathbb{E}\sup_{\phi \in \Phi(\bv z)}\frac1n \sum_{i=1}^n \phi_i \sigma_i \leq \frac{C_{\Omega,\Omega',d}^{1/2} \Gamma\left(\frac{d+3}{2}\right)}{\sqrt{n}}.
\end{equation}
Define a norm on the space of analytic functions $\phi:\Omega \to \mathbb{R}$
by $\|\phi\|_{\Phi} \coloneqq \sup_{z \in \Omega'} |\phi^{\ast}(z)|$.
By construction, if $\phi \in \Phi$, then $\|\phi\|_{\Phi} \leq 1$, and if $\|\phi\|_{\Phi} < \infty$, then $\phi / \|\phi\|_{\Phi} \in \Phi$. With
this notation, our primary assumption is as follows.
\begin{assumption}
\label{ass:Main}
There exist points $z_1,\dots,z_n \in \Omega$ and a function $q$ such that $q_i(\bv x) = q(\bv x; z_i)$ for each $i=1,\dots,n$, and there exists $\Omega'$ with $\Omega~\subsetneq~\Omega'~\subset~\mathbb{C}^d$ for which
\begin{itemize}
\item $\|q\|_{\Phi} \coloneqq \sup_{\bv x \in [0,1]^n}\|q(\bv x;\cdot)\|_{\Phi} < +\infty$;
\item $\|Dq\|_{\Phi} \coloneqq \sum_{i=1}^n\sup_{\bv x \in [0,1]^n}\|\partial_{x_i} q(\bv x; \cdot)\|_{\Phi} < +\infty$.
\end{itemize}
\end{assumption}
While we expect Assumption \ref{ass:Main} to hold for many practical finite
spin systems (as most of them are constructed to occupy some region of space),
it is not altogether trivial to verify.
An obvious sufficient condition to guarantee that $\|\phi\|_{\Phi} < \infty$ is to impose that
$\phi$ has an analytic continuation that is entire on $\mathbb{C}^d$. 
For example, if $z_1,\dots,z_n \in [-1,1]^d$ and
\[
q^+_i(x) = \frac{\lambda}n \sum_{j=1}^n \kappa(z_i - z_j)x_j,\qquad
q^-_i(x) = \mu,\qquad i=1,\dots,n,
\]
where $\lambda,\mu > 0$, and $\kappa(x) = e^{-a\|x\|^2}$ is a Gaussian kernel with bandwidth $1/a$,
then Assumption \ref{ass:Main} is satisfied for any choice of $\Omega'$. Taking $\Omega' = \{z \in \mathbb{C} \; : \; |z| \leq R\}^d$ for $R > 1$, then
\[
\|q\|_{\Phi} \leq \lambda e^{4 a d R^2} + \mu,\qquad \|Dq\|_{\Phi} \leq \lambda e^{4 a d R^2}.
\]
This form arises in connection with the \emph{Hanski incidence function model} \cite{mcvinish2014limiting}. On the other hand, consider the Ising model with Gaussian Kac potentials \cite{mourrat2017convergence}, given by
\begin{equation}
\label{eq:IsingKacGaussian}
q_i^\pm(x) = \frac12\left[1 \pm \tanh\left(\frac{\beta}{n} \sum_{j=1}^n e^{-a\|z_i-z_j\|^2} (2 x_j - 1)\right)\right],
\end{equation}
where each $z_1,\dots,z_n \in [-1,1]^d$. Since $\tanh(x)$ is not entire, we must take care to choose $\Omega'$, $a$ and $\beta$ appropriately to avoid its poles. 
Fortunately, choosing $\Omega$ to be a sufficiently thin open set on $\mathbb{C}^d$ encompassing $[-1,1]^d$, there will always exist an $\Omega'$ for which Assumption \ref{ass:Main} holds.

\subsection{Main result}
Now that appropriate assumptions on the rate functions $q_i$ have been established,
we can proceed with our main result on the exact error asymptotics of the Euler method.
Of course, quantifying a rate of convergence requires an appropriate metric --- a natural choice, like the strong error bounds in the previous section, would be the $L^1$ norm. Unfortunately, we have found this to require extraordinary control over the quadratic variation of the errors $\eta_i(t) - \eta^{\delta}_i(t)$. Instead, we shall do so
under the classic metrisation of convergence in probability. For any random variable $X$, we let $\bar{\mathbb{E}}X \coloneqq \mathbb{E}[X \wedge 1]$, recalling that for random variables $X,X_1,X_2,\dots$, $\bar{\mathbb{E}}|X_n - X| \to 0$ if and only if $X_n \cp X$. We shall make use of the fact that the truncated $L^1$ norm satisfies
\begin{equation}
\label{eq:KyFanBounds}
\bar{\mathbb{E}}|X| \leq 2 \,\inf\{\epsilon > 0 \; : \; \mathbb{P}(|X| > \epsilon) < \epsilon\}.
\end{equation}
Also, to facilitate the proof, it will be necessary to define the following:
\[
\tilde{\Gamma}_n = \sum_{i=1}^n \max_{j,k=1,\dots,n} \|\partial_j \partial_k q_i\|_{\infty},\qquad
\theta_n = \sum_{\substack{i,k=1\\i\neq k}}^n \max_{j=1,\dots,n} \|\partial_j \partial_k^2 q_i\|_{\infty}.
\]
With this in tow, we present our main result in Theorem \ref{thm:EulerExact}.
\begin{theorem}
\label{thm:EulerExact}
Suppose that Assumption \ref{ass:Main} holds and $n\delta \geq 1$. There exists a polynomial $\mathcal{P}$ in
\[
T,\enskip\|q\|_{\Phi},\enskip\|Dq\|_\Phi,\enskip \|D^{\ast}q\|_{F},\enskip n\tilde{\Gamma}_n,\enskip n^{1/2} \theta_n,
\]
such that
\begin{multline*}
\bar{\mathbb{E}}\sup_{\phi\in\Phi}\sup_{t\in[0,T]}\left|\frac{1}{n\delta}\sum_{i=1}^{n}\phi_{i}(\eta_{i}(t)-\eta^{\delta}_{i}(t)-\delta\mathcal{E}_{i}(t))\right| \\\leq \mathcal{P} e^{6T(\|q\|_{\Phi} + \|Dq\|_{\Phi})}\left(\frac{\log(1 + n\delta)}{(n\delta)^{1/3}} + \delta\right).
\end{multline*}
\end{theorem}
The proof of Theorem \ref{thm:EulerExact} is quite long, comprising Appendix \ref{sec:EulerExactProof}. For mean-field-type models satisfying Assumption \ref{ass:Main}, we can see that $\mathcal{P}$ is necessarily bounded independently of $n$ and $\delta$. However, we suspect that the order of approximation in $n\delta$ of this bound can be improved with better control over the quadratic variation of the errors $\eta_i(t) - \eta^{\delta}_i(t)$. Consider, for instance, that if $\rho(t)$ has a stable equilibrium $\rho_{\infty} = \lim_{t\to\infty} \rho(t)$, then $n^{-1}\sum_{i=1}^n \phi_i \mathcal{E}_i(t) \to 0$ as $t \to \infty$. Furthermore, if $\eta(t)$ starts near its (quasi-)stationary regime with each $\eta_i(0)$ taken to be independent with $\mathbb{P}(\eta_i(0) = \rho_{\infty,i})$ for each $i=1,\dots,n$, then $\mathcal{E}(t) = 0$ and so Theorem \ref{thm:EulerExact} suggests that
\[
\bar{\mathbb{E}}\sup_{\phi\in\Phi}\sup_{t\in[0,T]}\left|\sum_{i=1}^n \phi_i (\eta_i(t) - \eta_i^{\delta}(t))\right| = \bigO((n\delta)^{2/3} + n\delta^2),
\]
excluding logarithmic terms. On the other hand, assuming that the error terms $\mathcal{E}_i$ comprise all first-order errors in the Euler approximation, we expect the Euler and midpoint approximations to perform comparably in this scenario. In the next section, we shall conduct numerical experiments to support this hypothesis. Based on this, we should expect the optimal rate of convergence to be at most $\bigO((n\delta)^{-1/2} + \delta)$, excluding logarithmic terms.

\section{Numerical experiments}
\label{sec:Numerics}
To complement the results obtained thus far and assess their relevance in practice, we conduct numerical experiments comparing the performance of the Doob-Gillespie algorithm to the Euler and midpoint approximations. 
We shall begin by assessing accuracy of the approximations by simulating the coupled processes 
(\ref{model:eq4}), 
(\ref{model:eq5}) and (\ref{model:eq6}) with respect to the same Poisson processes. 
Of course, this negates the computational advantages of the Euler and midpoint approximations, so we shall forego comparisons of computation time between the methods for now. 

For the sake of illustration, we restrict ourselves to the one--dimensional case (with site space $\{1,\dots,n\}$), considering a transition rate function constructed according to an Gaussian convolutional kernel:
\begin{equation}
\label{eq:GaussConvExample}
q_i^+(\bv x) = \frac{2 \sigma}{n\sqrt{\pi}}\sum_{j=1}^n \exp\left[-\left(\frac{\sigma(i - j)}{n}\right)^2\right] x_j,
\end{equation}
for some scale parameter $\sigma > 0$. The extinction rates are chosen uniformly $q_i^-(\bv x) \equiv 1$. The corresponding spin system is of mean-field-type, and satisfies Assumption \ref{ass:Main}. As $n\to\infty$, the solution $\bv \rho_{\infty}$ to $q_i(\bv \rho_{\infty}) = 0$ for $i=1,\dots,n$ becomes increasingly uniform $\bv \rho_{\infty} \equiv \rho$. The quasi-stationary distribution of $\eta$ is approximately a product of $n$ Bernoulli random variables with probability $1/2$; in fact, for each $i=1,\dots,n$, 
\[
q_i^+(\rho\bv 1) = 2 \rho + \bigO(e^{-\sigma^2}) + \bigO(n^{-1}),\qquad \mbox{as } \sigma,n\to \infty,
\]
Since (\ref{eq:GaussConvExample}) is of convolutional form, the corresponding potentials can be computed efficiently using the univariate fast Fourier transform.

Figure \ref{fig:EulerMidpointCompare} presents a realisation of the Euler and midpoint approximations and their evolution in time for the same spin system, coupled together with the same Poisson processes. Locations of errors from an exact Doob-Gillespie simulation for the same process are also shown. Here, the process does not start near quasi-stationarity, and the midpoint approximation displays significantly fewer errors than the Euler approximation throughout the duration of the simulation.  

\begin{figure}
\includegraphics[width=\textwidth]{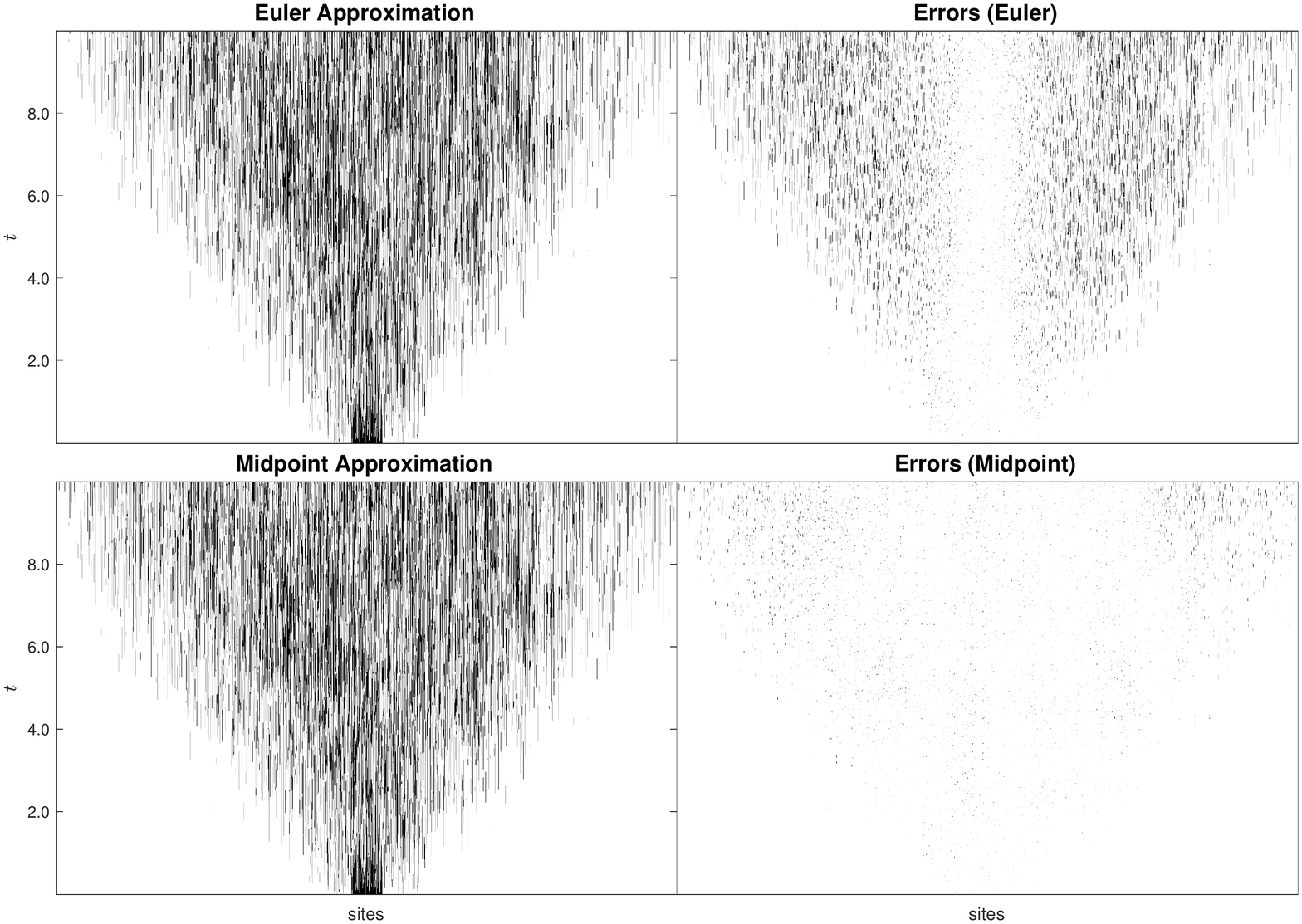}
\caption{
\label{fig:EulerMidpointCompare}
Simulation of the Euler (top) and midpoint (bottom) approximations for a finite spin system with $n=2000$ sites, $q_i^+$ given by (\ref{eq:GaussConvExample}) with $\sigma = 20$, and $q_i^- \equiv 1$ for $i=1,\dots,n$, coupled together with the same Poisson processes, and with step size $\delta = n^{-1/3}$. For each process, the realisation is presented on left with presence ($\eta_i(t) = 1$) denoted by a black site and absence ($\eta_i(t) = 0$) a white site. Locations of errors compared with the coupled Doob-Gillespie simulation are presented on right in black.}
\end{figure}

To more precisely compare the accuracy of the two methods, we can measure the
largest proportion of errors encountered up to that point in time. For the Euler and midpoint approximations, this is given by
\begin{equation}
\label{eq:NumError}
\sup_{s \in [0,t]} \frac1n \sum_{i=1}^n |\eta_i(s) - \eta_i^{\delta}(s)|,\qquad \sup_{s \in [0,t]} \frac1n \sum_{i=1}^n |\eta_i(s) - \mathring{\eta}_i^{\delta}(s)|.
\end{equation}
These quantities are plotted in Figure \ref{fig:CummaxErrors}, whose left-hand side once again illustrates the improved accuracy for the midpoint approximation when initialised away from quasi-stationarity. In fact, we see improved accuracy for the midpoint approximation, even for a much larger step size (ten times larger, in this case). On the other hand, as the process tends towards quasi-stationarity in time, $\sup_{i=1,\dots,n}\mathcal{E}_i(t) \to 0$. As discussed earlier, in this regime, Theorem \ref{thm:EulerExact} would still imply increased accuracy for the midpoint approximation. However, this does not appear to be the case --- as seen in the right-hand plot of Figure \ref{fig:CummaxErrors}, when the process is started near quasi-stationarity, the accuracy of the Euler and midpoint approximations appear comparable. This supports the hypothesis that the rate of convergence in Theorem \ref{thm:EulerExact} may be further improved.
\begin{figure}
\includegraphics[width=0.45\textwidth]{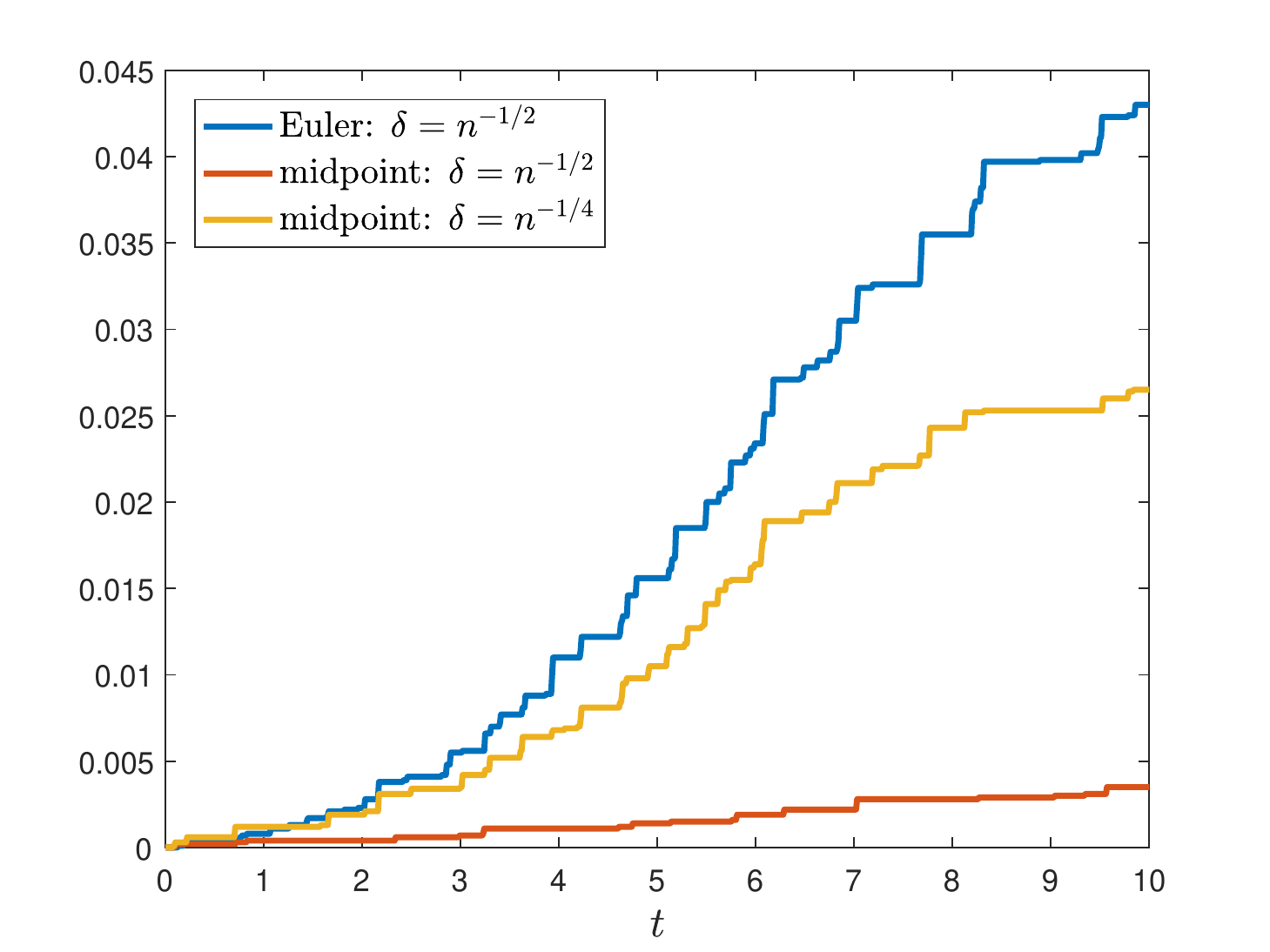}
\includegraphics[width=0.45\textwidth]{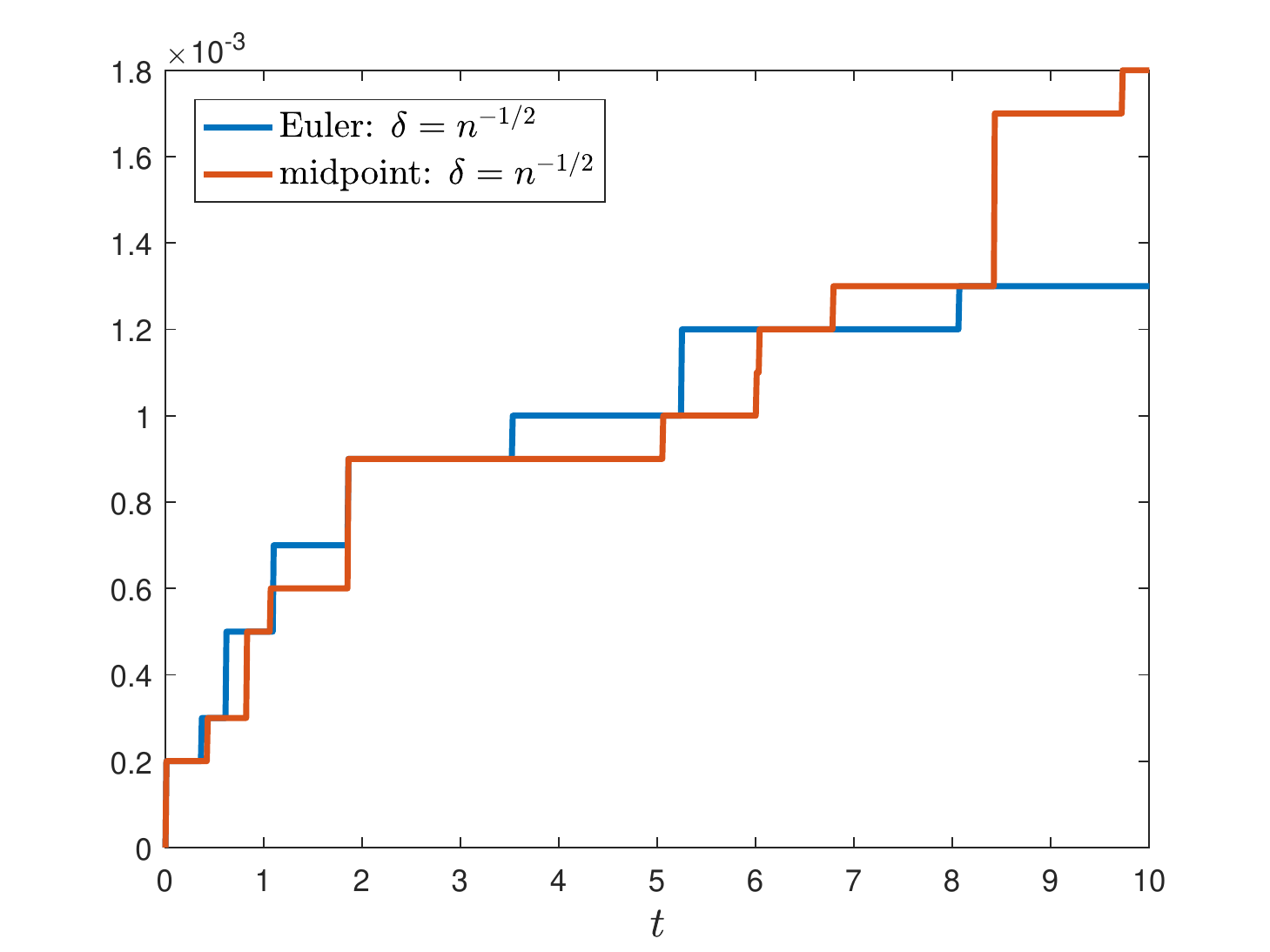}
\caption{
\label{fig:CummaxErrors}
Plots of the largest proportion of errors encountered (\ref{eq:NumError}) up to time $t$ for the Euler approximation with step size $\delta = n^{-1/2}$ (blue) and midpoint approximation with step sizes $\delta = n^{-1/2}$ (red) and $\delta = n^{-1/4}$ (yellow) for a spin system with $n=10^4$ sites, $q_i^+$ given by (\ref{eq:GaussConvExample}) with $\sigma = 20$, and $q_i^- \equiv 1$ for $i=1,\dots,n$, coupled together with the same Poisson processes. The process is started with 10\% of its sites initially occupied (left) and near the quasi-stationary distribution with $\mathbb{P}(\eta_i(0) = 1) = 1/2$ for all $i=1,\dots,n$ (right).}
\end{figure}
Finally, to illustrate the precision of the exact asymptotics in Theorem \ref{thm:EulerExact}, Figure \ref{fig:NormalisedErrors} displays the average normalised errors $(n\delta)^{-1}\sum_{i=1}^n(\eta_i(t) - \eta_i^{\delta}(t))$ with the average predicted errors $n^{-1}\sum_{i=1}^n \mathcal{E}_i(t)$ for the Euler approximation. 

\begin{figure}
\includegraphics[width=0.45\textwidth]{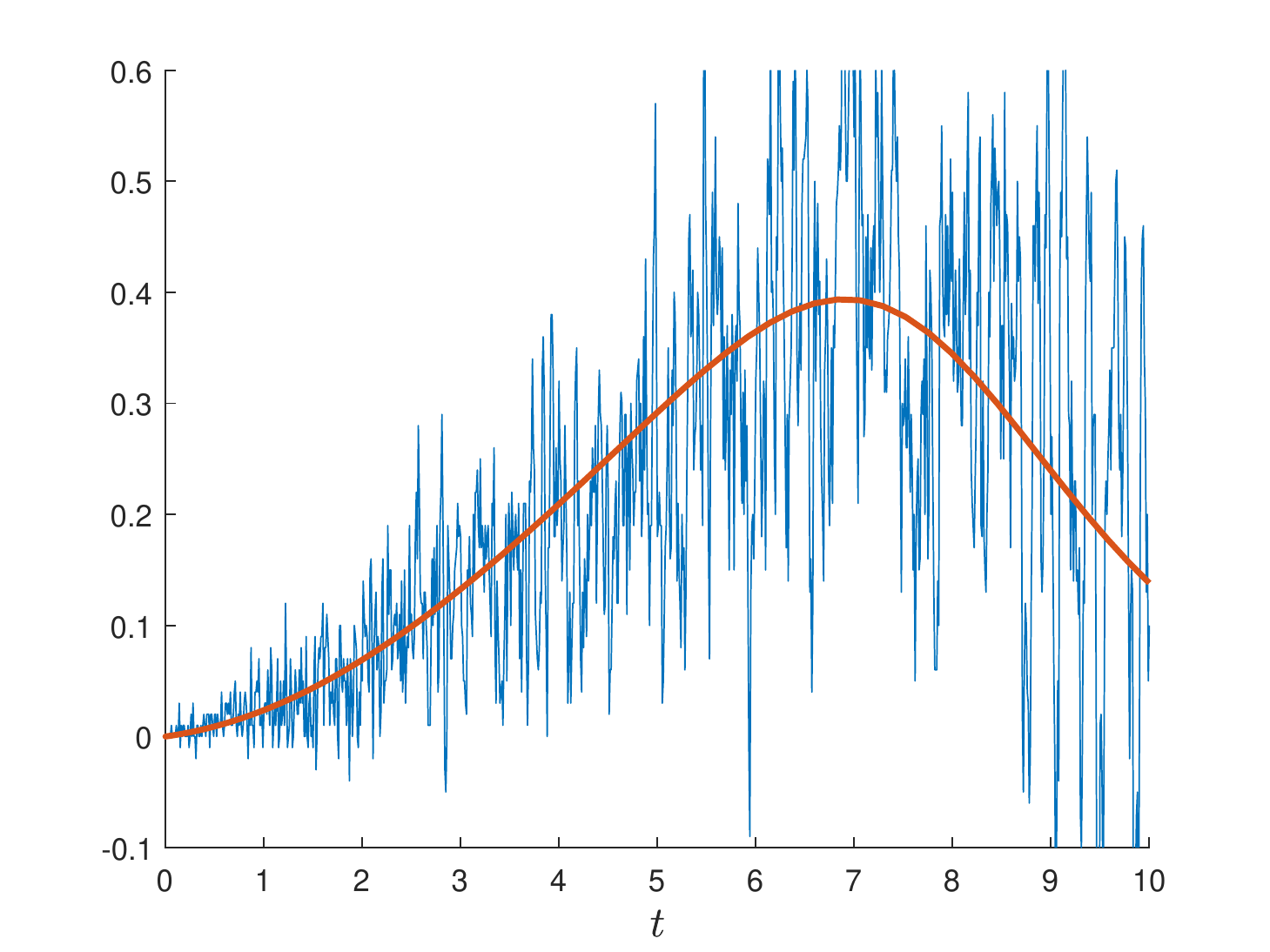}
\includegraphics[width=0.45\textwidth]{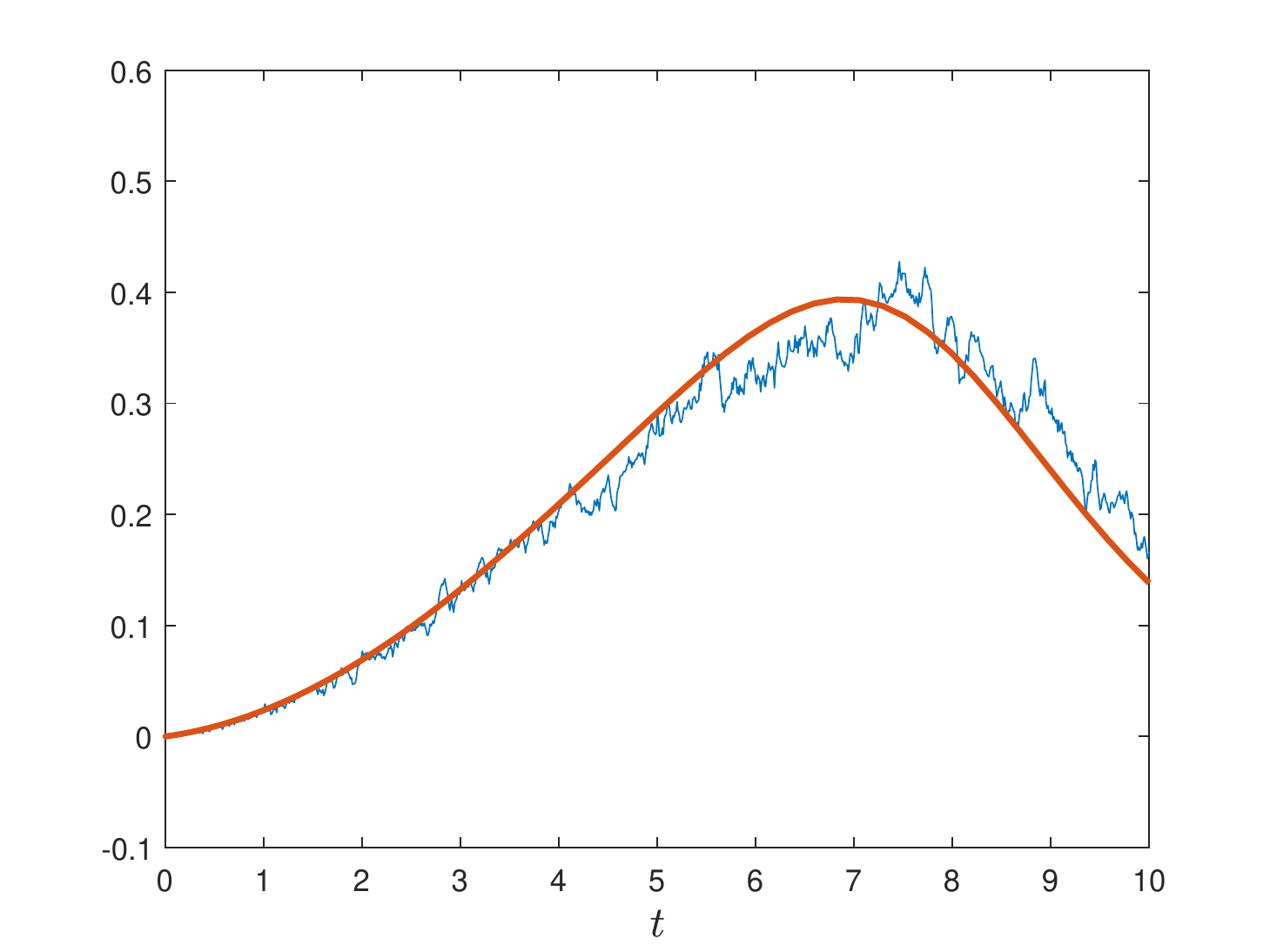}
\caption{
\label{fig:NormalisedErrors}
Comparison of normalised errors
$(n\delta)^{-1}\sum_{i=1}^n (\eta_i(t)-\eta^{\delta}_i(t))$ for the Euler approximation (blue) to
average predicted first-order error $n^{-1}\sum_{i=1}^n \mathcal{E}_i(t)$ (red) for
a finite spin system with $q_i^+$ given by (\ref{eq:GaussConvExample}) with $\sigma = 20$, and $q_i^- \equiv 1$ for $i=1,\dots,n$, with 10\% of its sites initially occupied. The number of sites $n$ and the step size $\delta$ are taken to be $n=10^4$ and $\delta = n^{-1/2}$ (left); and $n = 10^5$ and $\delta = n^{-1/4}$ (right). }
\end{figure}

\begin{figure}
\centering\includegraphics[width=.8\textwidth]{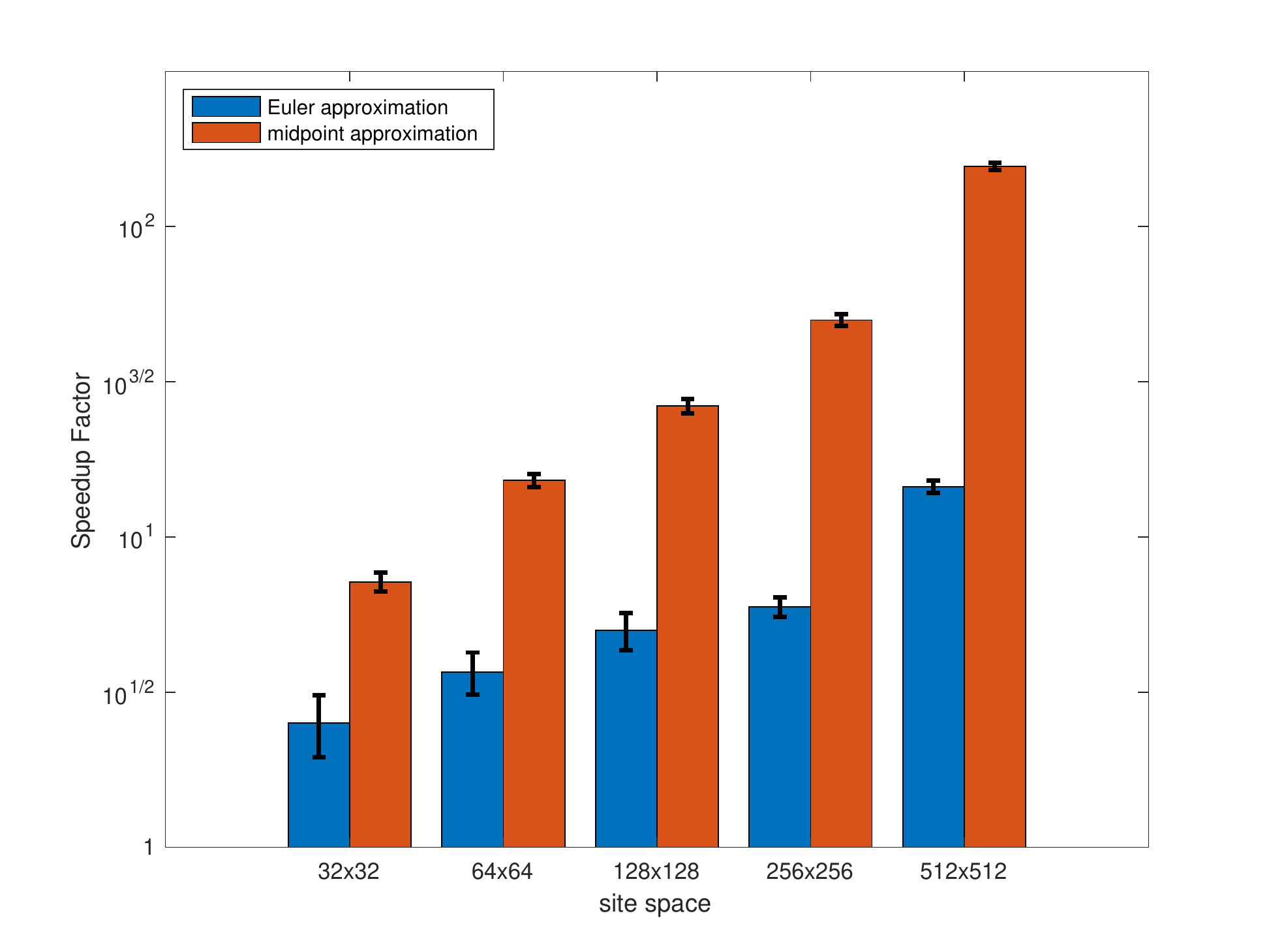}
\caption{Average speedup factors for Algorithms \ref{alg:EulerApprox} (blue) and \ref{alg:MidpointApprox} (red) over the Doob--Gillespie algorithm for the Ising--Kac model over two-dimensional site spaces of size $2^m \times 2^m$ for $m \in \{5,\dots,9\}$, with initial states $\eta_i(0)$ generated independently with $\mathbb{P}(\eta_i(0) = 1) = 1/2$. \label{fig:Speedup}}
\end{figure}
Turning now to computation time, in Figure \ref{fig:Speedup}, average speedup factors over the time interval $[0,T]$ for the Euler and midpoint approximations (Algorithms \ref{alg:EulerApprox} and \ref{alg:MidpointApprox}, respectively) over the Doob--Gillespie algorithm are presented. Here, we consider the Ising--Kac model over a two-dimensional site space $S = \{1,\dots,n\}\times\{1,\dots,n\}$ with $q_i^+$ and $q_i^-$ given by (\ref{eq:IsingKacGaussian}) with $\beta = 1$ and $a = 40 n^{-1}$, which we have shown to satisfy Assumption \ref{ass:Main}. In all cases, the process is initialised with an initial distribution of $\mathbb{P}(\eta_i(0) = 1) = 1/2$.
The step sizes for the Euler and midpoint approximations are chosen to be of the order of the maximal competitive step size; that is, $\delta = n^{-1/2}$ for the Euler approximation, and $\delta = n^{-1/4}$ for the midpoint approximation. As was shown in Figure \ref{fig:CummaxErrors}, the accuracy of the midpoint approximation is competitive to that of the Euler approximation for these choices of step size. Despite the roughly-doubled computation time for the midpoint approximation over the Euler approximation for the same step size, for large site spaces, the significantly smaller step size for the midpoint approximation sees greatly improved performance over both the Euler and Doob-Gillespie algorithms. Altogether, Figures \ref{fig:CummaxErrors} and \ref{fig:Speedup} suggest that, among the three methods discussed, the midpoint approximation is the better choice for finite spin systems not near (quasi-)stationarity. On the other hand, since the Euler approximation is universally faster to compute than the midpoint approximation for the same step size, Figure \ref{fig:CummaxErrors} suggests that the Euler approximation may be preferred for finite spin systems near (quasi-)stationarity. 

%

\appendix

\section{Independent site approximation}
\label{sec:IndepSite}
To enable the application of standard techniques, as in \cite{BMP:15,HMP:18}, we compare $\eta(t)$
with an independent site approximation $\omega(t)$ with transition rates
\begin{align*}
\omega_{i}(t) & \to \omega_{i}(t)+1\qquad\mbox{at rate }q^+_{i}(\rho(t))(1-\omega_{i}(t))\\
\omega_{i}(t) & \to \omega_{i}(t)-1\qquad\mbox{at rate }q^-_{i}(\rho(t))\omega_{i}(t),
\end{align*}
for each $i=1,\dots,n$, and where $\rho(t)$ is defined in (\ref{ODE:eq1}). Because each element of $\omega(t)$ is independent
for all time, for $1\leq r\leq2$, using the
method of bounded differences \cite[Corollary 3.2]{Boucheron2013}
and the Lindeburg argument \cite[eq. (2.5)]{HMP:18}, we can estimate for any smooth function $f$,
\begin{equation}
\left\lVert f(\omega(t))-f(\rho(t))\right\rVert _{r}
 \leq\frac{1}{2}\left(\sum_{j=1}^{n}\left\lVert \partial_{j}f\right\rVert _{\infty}^{2}\right)^{1/2}+\frac{1}{2}\sum_{j=1}^{n}\left\lVert \partial_{j}^{2}f\right\rVert _{\infty}.\label{eq:WFuncApprox}
\end{equation}
To compare $\eta(t)$ and $\omega(t)$, we couple them together using the
basic coupling for spin systems (see \cite[Section III.1]{Liggett:2005}). Let
$J_{i}(T)\coloneqq\sup_{t\in[0,T]}\ind\{\eta_{i}(t)\neq \omega_{i}(t)\}$
and define $\bar{J}(T)\coloneqq n^{-1}\sum_{i=1}^{n}J_{i}(T)$. Now,
since $\mathcal{X}=(\eta,\omega,J,\rho)$ is a Feller process, denoting by $L_{\mathcal{X}}$
the infinitesimal generator of $\mathcal{X}$, we may define the process
$\mathcal{J}(t)$ by
\begin{align}
\mathcal{J}(t) & \coloneqq L_{\mathcal{X}}\bar{J}(\eta(t),\omega(t),J(t),\rho(t))\nonumber \\&=\frac{1}{n}\sum_{i=1}^{n}L_{\mathcal{X}}J_{i}(\eta(t),\omega(t),J(t),\rho(t))\nonumber \\
 & =\frac{1}{n}\sum_{i=1}^{n}(1-J_{i}(t))\{(1-\omega_{i}(t))|q^+_{i}(\eta(t))-q^+_{i}(\rho(t))|\label{eq:JDerivProc}\\&\qquad\qquad+\omega_{i}(t)|q^-_{i}(\eta(t))-q^-_{i}(\rho(t))|\}\nonumber,
\end{align}
and by Dynkin's formula, since $\bar{J}(0)=0$, there exists a martingale
$M$ such that
\begin{equation}
\bar{J}(t)=M(t)+\int_{0}^{t}\mathcal{J}(s)\mathd s,\qquad t\geq0.\label{eq:DynkinJBar}
\end{equation}
\begin{lemma}
\label{lem:JWBarBound}
For any $T\geq0$, 
\begin{align}
\mathbb{E}\bar{J}(T) & \leq n^{-1} (\|D^{\ast}q\|_{2,1}+\gamma_n)Te^{2 T\|D^{\ast}q\|_{1}}\label{eq:JBarBound1}\\
\|\bar{J}(T)\|_{2} & \leq 2 n^{-1} (1+(\|D^{\ast}q\|_{2,1}+\gamma_n) T) e^{3T\|D^{\ast}q\|_{1}}. \label{eq:JBarBound2}
\end{align}
\end{lemma}
\begin{proof}
Combining (\ref{eq:JDerivProc}) and (\ref{eq:DynkinJBar}), 
\[
\frac{\mathd}{\mathd t}\mathbb{E}\bar{J}(t)=\mathbb{E}\mathcal{J}(t)\leq\frac{1}{n}\sum_{i=1}^{n}\mathbb{E}|q^+_{i}(\eta(t))-q^+_{i}(\rho(t))|+\mathbb{E}|q^-_{i}(\eta(t))-q^-_{i}(\rho(t))|,
\]
and so applying (\ref{eq:WFuncApprox}) with the Mean Value Theorem
gives
\begin{equation}
\mathbb{E}\mathcal{J}(t)\leq 2 \|D^{\ast}q\|_{1}\mathbb{E}\bar{J}(t)+ n^{-1} (\|D^{\ast}q\|_{2,1}+\gamma_n) \label{eq:JDerivL1}
\end{equation}
and so
\[
\mathbb{E}\bar{J}(t)\leq 2 \|D^{\ast}q\|_{1}\int_{0}^{t}\mathbb{E}\bar{J}(s)\mathd s +  t n^{-1} (\|D^{\ast}q\|_{2,1}+\gamma_n).
\]
Equation (\ref{eq:JBarBound1}) now follows from Gronwall's inequality.
To show (\ref{eq:JBarBound2}), once again from (\ref{eq:DynkinJBar}),
\begin{equation}
\|\bar{J}(t)\|_{2}\leq\|M(t)\|_{2}+\int_{0}^{t}\|\mathcal{J}(s)\|_{2}\mathd s,\qquad t\geq0,\label{eq:JBar2Decomp}
\end{equation}
and from (\ref{eq:JDerivProc}),

\[
\|\mathcal{J}(t)\|_{2}\leq \left\lVert \frac{1}{n}\sum_{i=1}^{n}q^+_{i}(\eta(t))-q^+_{i}(\rho(t))\right\rVert_{2}+\left\lVert\frac{1}{n}\sum_{i=1}^{n}q^-_{i}(\eta(t))-q^-_{i}(\rho(t))\right\rVert_{2},
\]
whereby (\ref{eq:WFuncApprox}) with the Mean Value Theorem once again
implies
\begin{equation}
\|\mathcal{J}(t)\|_{2}\leq 2 \|D^{\ast}q\|_{1}\|\bar{J}(t)\|_{2} + n^{-1}(\|D^{\ast}q\|_{2,1}+\gamma_n),\qquad t\geq0.\label{eq:JDerivL2}
\end{equation}
It only remains to bound $\|M(t)\|_{2}$. However, since $M(0)=0$, $\mathbb{E}M(t)^{2}= \mathbb{E} [M]_{t}$, where $ [ M ]_{t}$ is the quadratic variation
of $M$. As $\bar{J}$ moves only by jumps of size $ n^{-1} $, it follows $[M]_{t} = n^{-1} \bar{J}(t) $. 
Therefore, by (\ref{eq:JBarBound1}),
\begin{align}
\left\lVert M(T)\right\rVert _{2} & \leq n^{-1} (\|D^{\ast}q\|_{2,1}+\gamma_n) ^{1/2} T^{1/2}e^{T\|D^{\ast}q\|_{1}} \\
& \leq n^{-1} (1 + (\|D^{\ast}q\|_{2,1} + \gamma_n)T) e^{T\|D^{\ast}q\|_1}.\label{eq:JBarMartBound}
\end{align}
Combining
(\ref{eq:JBar2Decomp}), (\ref{eq:JDerivL2}), and (\ref{eq:JBarMartBound}) yields 
\[
\|\bar{J}(t)\|_{2}\leq 2\|D^{\ast}q\|_{1}\int_{0}^{t}\|\bar{J}(s)\|_{2}\mathd s+ 2 n^{-1} (1+(\|D^{\ast}q\|_{2,1}+\gamma_n)t ) e^{t\|D^{\ast}q\|_{1}},
\]
from whence (\ref{eq:JBarBound2}) follows by Gronwall's inequality.
\end{proof}
Similarly, for the Euler tau-leaping process, we compare $\eta^\delta(t)$ with
its independent site approximation $\omega^{\delta}(t)$ with transition rates
\begin{align*}
\omega^{\delta}_{i}(t) & \to \omega^{\delta}_{i}(t)+1\qquad\mbox{at rate }q^+_{i}(\rho^{\delta} \circ \chi (t))(1-\omega^{\delta}_{i}(t))\\
\omega^{\delta}_{i}(t) & \to \omega^{\delta}_{i}(t)-1\qquad\mbox{at rate }q^-_{i}(\rho^{\delta} \circ\chi (t))\omega^{\delta}_{i}(t),
\end{align*}
for each $i=1,\dots,n$, where
\[
\frac{d}{dt} \rho^{\delta}_i (t) = (1- \rho^{\delta}_{i}(t)) q^+_{i}(\rho^{\delta} \circ \chi(t)) - \rho^{\delta}_i(t) q^-_{i}(\rho^{\delta} \circ \chi(t))
\]
As before, each element of $\omega^{\delta}(t)$ is independent for all time, $\mathbb{E}\omega^{\delta}(t) = \rho^{\delta}(t)$, and
\begin{equation}
\label{eq:VFuncApprox}
\|f(\omega^{\delta}(t))-f(\rho^{\delta}(t))\|_r \leq\frac{1}{2}\left(\sum_{j=1}^{n}\| \partial_{j}f\| _{\infty}^{2}\right)^{1/2}+\frac{1}{2}\sum_{j=1}^{n}\| \partial_{j}^{2}f\| _{\infty}.
\end{equation}
Define $J_i^\delta(T)\coloneqq \sup_{t\in[0,T]}\ind\{\eta^\delta_i(t)\neq \omega^\delta_i(t)\}$
and $\bar{J}_\delta(T)\coloneqq n^{-1}\sum_{i=1}^n J_i^\delta(T)$.
Neither $\eta^\delta$ nor $\omega^\delta$ are Markov processes, however, they are Markov over each
time interval $[k\delta,(k+1)\delta)$ for $k \geq 0$. Therefore, for any $k \geq 0$, we can couple $\eta^\delta$ and $\omega^\delta$ together using the basic coupling, and
form a Feller process $\mathcal{Z}_k = (\eta^\delta,\omega^\delta,J^\delta)$ over $[k\delta,(k+1)\delta)$, conditioned
on $\eta^\delta(k\delta)$, $\omega^\delta(k\delta)$, and $J^\delta(k\delta)$. Proceeding as before, by Dynkin's formula,
there is a martingale $M_k$ such that
\begin{equation}
\label{eq:JVDerivL1}
\bar{J}_\delta(k\delta+t) = \bar{J}_\delta(k\delta) + M_k^\delta(t) + \int_{k\delta}^{k\delta+t} \mathcal{J}^\delta(s) \mathd s,\qquad 0 \leq t < \delta,
\end{equation}
where
\begin{multline*}
\mathcal{J}^\delta(t) = \frac1n \sum_{i=1}^n(1-J_i^\delta(t))(1-\omega^\delta_i(t))|q^+_i(\eta^\delta\circ\chi(t))-q^+_i(\rho^\delta\circ\chi(t))|+\\
\frac1n \sum_{i=1}^n(1-J_i^\delta(t))\omega^\delta_i(t)|q^-_i(\eta\circ\chi(t))-q^-_i(\rho^\delta\circ\chi(t))|.
\end{multline*}

\begin{lemma}
\label{lem:JVBarBound}
For any $T\geq0$, 
\begin{align}
\mathbb{E}\bar{J}_\delta(T) & \leq  n^{-1} (\|D^{\ast}q\|_{2,1}+\gamma_n)Te^{2 T\|D^{\ast}q\|_{1}}\label{eq:JVBarBound1}\\
\|\bar{J}_\delta(T)\|_{2} & \leq 2 n^{-1} (1 + (\|D^{\ast} q\|_{2,1} + \gamma_n)T) e^{3 T \|D^{\ast}q\|_1}.\label{eq:JVBarBound2}
\end{align}
\end{lemma}

\begin{proof}
As in the proof of Lemma \ref{lem:JWBarBound}, by the Mean Value Theorem and (\ref{eq:VFuncApprox}),
\begin{equation}
\label{eq:FancyJExp}
\mathbb{E}\mathcal{J}^\delta(t) \leq 2 \|D^\ast q\|_1 \mathbb{E}\bar{J}_\delta(\chi(t))
+ n^{-1}(\|D^{\ast}q\|_{2,1} + \|D^2 q\|_{1,1}),\qquad t \geq 0,
\end{equation}
and so by direct integration, (\ref{eq:JVDerivL1}) implies that for $0 \leq t < \delta$,
\begin{equation}
\label{eq:JL1Inter}
\mathbb{E}\bar{J}_\delta(k\delta+t) \leq (1 + 2 t\|D^{\ast}q\|_1) \mathbb{E}\bar{J}_\delta(k\delta) +  tn^{-1} (\|D^{\ast}q\|_{2,1} + \gamma_n).
\end{equation}
Since $\mathbb{E}\bar{J}_\delta(t)$ is continuous in $t$, taking $t \to \delta^{-}$, this implies
\[
\mathbb{E}\bar{J}_\delta((k+1)\delta) \leq (1 + 2 \delta\|D^{\ast}q\|_1) \mathbb{E}\bar{J}_\delta(k\delta) + \delta n^{-1} (\|D^{\ast}q\|_{2,1} + \gamma_n).
\]
Performing induction over $k$, with $\bar{J}_\delta(0) = 0$ gives
\[
\mathbb{E}\bar{J}_\delta(k\delta) = n^{-1} (\|D^{\ast}q\|_{2,1} + \gamma_n) k\delta e^{2 k\delta \|D^{\ast} q\|_1}.
\]
Together with (\ref{eq:JL1Inter}), this implies (\ref{eq:JVBarBound1}). 
We shall now prove (\ref{eq:JVBarBound2}). By iterating equation (\ref{eq:JVDerivL1}), for any $k\geq 0$
and $0\leq t < \delta$,
\[
\bar{J}_\delta(k\delta + t) = \int_0^{k\delta + t} \mathcal{J}^\delta(s) \mathd s + M_k^\delta(t) + \sum_{l=0}^{k-1}M_l^\delta(\delta).
\]
It can be seen that $\{M_1^\delta(\delta),\dots,M_{k-1}^\delta(\delta),M_k^\delta(t)\}$ forms a martingale difference sequence relative
to the increasing sequence of $\sigma$-algebras $\mathcal{F}_k = \sigma(\eta^\delta(t),\omega^\delta(t),J^\delta(t),\ 0 \leq t \leq k\delta)$.
Therefore,
\begin{equation}
\label{eq:JVMartSum}
\left\lVert M_k^\delta(t) + \sum_{l=0}^{k-1} M_l^\delta(\delta)\right\rVert_2^2 = \mathbb{E}M_k^\delta(t)^2 + \sum_{l=0}^{k-1}
\mathbb{E}M_l^\delta(\delta)^2.
\end{equation}
Since $\mathbb{E}M_k^\delta(t)^2 = [ M_k^\delta]_t  = n^{-1} \left(\mathbb{E} \bar{J}_{\delta}(k\delta+t) - \mathbb{E}\bar{J}_\delta(k\delta) \right) $ for each $k$ and $0 \leq t \leq \delta$, (\ref{eq:JVBarBound1})  implies
\begin{align*}
\left\lVert M_k^\delta(t) + \sum_{l=0}^{k-1} M_l^\delta(\delta)\right\rVert_2 & \leq n^{-1}(\|D^{\ast} q\|_{2,1} + \gamma_n)^{1/2} (k\delta + t)^{1/2} e^{ (k\delta + t)\|D^{\ast}q\|_1} \\
& \leq n^{-1}(1+(\|D^{\ast} q\|_{2,1} + \gamma_n)(k\delta + t))  e^{(k\delta + t)\|D^{\ast}q\|_1} .
\end{align*}
On the other hand, equation (\ref{eq:VFuncApprox}) implies that
\[
\|\mathcal{J}^\delta(t)\|_2 \leq 2 \|D^{\ast}q\|_1 \|\bar{J}_\delta(\chi(t))\|_2
+ n^{-1}(\|D^{\ast}q\|_{2,1} + \gamma_n),\qquad t \geq 0.
\]
which altogether implies that for any $k\geq 0$ and $0\leq t < \delta$, 
\begin{multline*}
\|\bar{J}_\delta(k\delta + t)\|_2 \leq 2 \delta \|D^{\ast} q\|_1 \sum_{l=0}^{k-1} \|\bar{J}_\delta(l \delta)\|_2 + 2 t \|D^{\ast}q\|_1 \|\bar{J}_V(k \delta)\|_2 \\
+2 n^{-1}(1 + (\|D^{\ast} q\|_{2,1} + \gamma_n)(k\delta + t)) e^{(k\delta + t)\|D^{\ast}q\|_1}. 
\end{multline*}
By induction, we find that
\[
\|\bar{J}_\delta(k\delta)\|_2 \leq 2 n^{-1}(1 + (\|D^{\ast} q\|_{2,1} + \gamma_n)k \delta) e^{3 k\delta \|D^{\ast}q\|_1}
\]
which finally implies (\ref{eq:JVBarBound2}).
\end{proof}

As usual, the primary application of the independent site approximation is for obtaining a bound on the first- and second-order weak error between $\eta(t)$ and the deterministic approximation $\rho(t)$. A bound on the first-order weak error for vector-valued functions, which follows from
(\ref{eq:WFuncApprox}), (\ref{eq:VFuncApprox}), and Lemmas \ref{lem:JWBarBound} and \ref{lem:JVBarBound}, is presented in Lemma \ref{lem:JWBarBoundCor}.
\begin{lemma} \label{lem:JWBarBoundCor}
For any $f =(f_{1},\ldots,f_{m}) \in \mathcal{C}^{2}( [0,1]^{n} , \mathbb{R}^{m}) $ and any $ t > 0 $,
\begin{multline*}
\sum_{i=1}^{m} [\mathbb{E} \left| f_{i}(\eta(t)) - f_{i}(\rho(t)) \right| + \mathbb{E}|f_i(\eta^\delta(t)) - f_i(\rho^\delta(t))|] \\ \leq 2\|D f\|_{1}  (\|D^{\ast}q\|_{2,1}+\gamma_n)te^{2t\|D^{\ast}q\|_{1}} + \|Df\|_{2,1} + \sum_{i,j=1}^n \|\partial_j^2 f_i\|_{\infty}.
\end{multline*}
\end{lemma}
Following the same arguments as \cite[Proposition 9]{HMP:18}, using
Lemmas \ref{lem:JWBarBound} and \ref{lem:JVBarBound} in place of \cite[Lemma 7]{HMP:18}, we arrive
at a bound on the second-order weak error, presented in Lemma
\ref{lem:JWBarBoundSecond}. For brevity, we introduce the notation
\[
\mathcal{T}_n(f) = \sqrt{n(1 + \log n)}\left[\sqrt{n}\max_{j,k=1,\dots,n}\|\partial_j \partial_k f\|_{\infty} + \max_{j=1,\dots,n}\sum_{k=1}^n \|\partial_j \partial_k^2 f\|_{\infty}\right].
\]

\begin{lemma}
\label{lem:JWBarBoundSecond}
There is a universal constant $ C>0 $ such that, for any $ f \in \mathcal{C}([0,1]^{n},\mathbb{R}) $ and $ t > 0 $,
\begin{align*}
\mathbb{E} \left| f(\eta(t)) - f(\rho(t)) - \sum_{j=1}^{n} \partial_{j} f(\rho(t)) (\eta_{j}(t) - \rho_{j}(t)) \right|
&\leq C (1 + n\|\bar{J}\|_2^2) \mathcal{T}_n(f) \\
\mathbb{E} \left| f(\eta^\delta(t)) - f(\rho^\delta(t)) - \sum_{j=1}^{n} \partial_{j} f(\rho^\delta(t)) (\eta^\delta_{j}(t) - \rho^\delta(t)) \right|
&\leq C (1 + n\|\bar{J}_{\delta}\|_2^2) \mathcal{T}_n(f).
\end{align*}
\end{lemma}

\section{Proofs}
\label{sec:Proofs}
\subsection{Strong error analysis}
We begin by providing proofs of the strong error bounds presented in Theorems \ref{thm:EulerRate} and \ref{thm:MidPointRate}. 
\begin{proof}[Proof of Theorem \ref{thm:EulerRate}]
As will be a recurring theme for all of the main results to be presented, 
the method of proof is centered about
an application of Gronwall's inequality. In this case, the application is
fairly standard; first, by decomposing
\[
\sum_{i=1}^n \mathbb{E} |\eta_i(t) - \eta_i^\delta(t)| \leq \mathbb{E} (T_+^{(1)} + T_+^{(2)} + T_-^{(1)} + T_-^{(2)}),
\]
where each of the terms are given by
\begin{align*}
T_{+}^{(1)}&\coloneqq \sum_{i=1}^n\int_{0}^{t}\left|(1-\eta_{i}(s))q_{i}^+(\eta(s))-(1-\eta^\delta_{i}(s))q^+_{i}(\eta^{\delta}(s))\right| \mathd s \\
T_{+}^{(2)}&\coloneqq \sum_{i=1}^n\int_{0}^{t}\left|(1-\eta^{\delta}_{i}(s))(q^+_{i}(\eta^{\delta}(s))-q^+_{i}(\eta^{\delta}\circ\chi(s)))\right| \mathd s \\
T_{-}^{(1)}&\coloneqq \sum_{i=1}^n\int_{0}^{t}\left|\eta_{i}(s)q^-_{i}(\eta(s))-\eta^{\delta}_{i}(s)q^-_{i}(\eta^{\delta}(s))\right| \mathd s \\
T_{-}^{(2)}&\coloneqq \sum_{i=1}^n\int_{0}^{t}\left|\eta^{\delta}_{i}(s)(q_{i}^-(\eta^{\delta}(s))-q_{i}^-(\eta^{\delta}\circ\chi(s)))\right| \mathd s.
\end{align*}
Treating $T_+^{(1)}$ and $T_+^{(2)}$ (the arguments for $T_-^{(1)}$ and $T_-^{(2)}$ are similar), the Mean Value Theorem implies
\begin{align*}
T_+^{(1)} 
&\leq (\|q^+\|_{\infty} + \|Dq^+\|_1) \int_0^t \sum_{i=1}^n |\eta_i(s) - \eta^{\delta}_i(s)| \mathd s, \\
T_+^{(2)} &\leq \|Dq^+\|_1 \int_0^t \sum_{i=1}^n |\eta^{\delta}_i(s) - \eta^{\delta}_i \circ \chi(s)| \mathd s. 
\end{align*}
Since $|s - \chi(s)| \leq \delta$ for any $s \geq 0$, $\mathbb{E}|\eta^{\delta}_i(s) - \eta^{\delta}_i \circ \chi(s)|$ is bounded above by $2\|q\|_{\infty} \delta$.
Altogether, this implies
\begin{multline*}
\sum_{i=1}^n \mathbb{E}|\eta_i(t) - \eta^{\delta}_i(t)| \leq 2(\|q\|_{\infty} + \|D^{\ast}q\|_{1}) \int_0^t \sum_{i=1}^n \mathbb{E}|\eta_i(s) - \eta^{\delta}_i(s)| \mathd s \\ + 4\|q\|_{\infty} \|D^{\ast}q\|_1 n \delta t,
\end{multline*}
and applying Gronwall's inequality completes the proof.
\end{proof}

\begin{proof}[Proof of Theorem \ref{thm:MidPointRate}]
We proceed as in the proof of Theorem \ref{thm:EulerRate}, but now
\begin{align*}
T_+^{(2)} &\coloneqq \sum_{i=1}^n \left|\int_0^t (1 - \mathring{\eta}^{\delta}_i(s))(q^+_i(\mathring{\eta}^{\delta}(s)) - q^+_i(p^{\delta} \circ \mathring{\eta}^{\delta}\circ \chi(s)))\mathd s\right| \\
T_-^{(2)} &\coloneqq \sum_{i=1}^n \left|\int_0^t \mathring{\eta}^{\delta}_i(s)(q^-_i(\mathring{\eta}^{\delta}(s)) - q^-_i(p^{\delta} \circ \mathring{\eta}^{\delta}\circ \chi(s)))\mathd s\right|.
\end{align*}
To apply Gronwall's inequality as before, it is necessary to bound the expectation of these terms. In fact, it will suffice to consider $T_+^{(2)}$ alone, as the procedure for $T_-^{(2)}$ is virtually identical. Let \[U_i(s) \coloneqq~q^+_i(\mathring{\eta}^{\delta}(s)) - q^+_i(p^{\delta} \circ \mathring{\eta}^{\delta} \circ \chi(s)).\] Applying a first order Taylor expansion to $q^+_i$ about $ \mathring{\eta}^{\delta}\circ\chi(s) $, 
\begin{multline*}
U_i(s) 
=  \sum_{j=1}^{n} \partial_{j} q^+_{i}(\mathring{\eta}^{\delta}\circ\chi(s)) \left(\mathring{\eta}^{\delta}_{j}(s) - \mathring{\eta}^{\delta}_{j}\circ\chi(s) - \frac{\delta}{2} Q_{j}(\mathring{\eta}^{\delta}\circ\chi(s)) \right) \\ + R_{i,1} + R_{i,2},
\end{multline*}
where $ R_{i,1} $ and $ R_{i,2}$  are the remainders from the expansion of $ q^+_{i}(p^{\delta}\circ\mathring{\eta}^{\delta}\circ\chi(s)) $ and $ q^+_{i}(\mathring{\eta}^{\delta}(s)) $, respectively. Note that $ |R_{i,1}| \leq  \delta^2 \|q\|_{\infty}^2 \sum_{j,k=1}^n\|\partial_j\partial_k q^+_i\|_{\infty}$. For $R_{i,2}$, 
\begin{align*}
\mathbb{E} \left|R_{i,2}\right| & \leq \sum_{j,k=1}^{n}\|\partial_{j} \partial_{k} q^+_{i}\|_{\infty} \mathbb{E} \left|(\mathring{\eta}^{\delta}_{j}(s) - \mathring{\eta}^{\delta}_{j}\circ\chi(s))(\mathring{\eta}^{\delta}_{k}(s) - \mathring{\eta}^{\delta}_{k}\circ\chi(s)) \right|. 
\end{align*}
From (\ref{model:eq6}), the number of jumps in $ \mathring{\eta}^{\delta}_{i} $ on $[a,b] $ is stochastically dominated by a Poisson random variable with mean $  2 \|q\|_{\infty} (b-a) $ so  for each $j=1,\dots,n$,
\begin{equation*}
\mathbb{E}\left[ \left.|\mathring{\eta}^{\delta}_j(s) - \mathring{\eta}^{\delta}_j \circ \chi(s)|\; \right|\; \mathring{\eta}^{\delta} \circ \chi(s) \right]
\leq 1 - e^{-2\delta\|q\|_{\infty}} \leq 2\delta\|q\|_{\infty}.
\end{equation*}
Since the $ \mathring{\eta}^{\delta}_{j}(s) $ are independent conditioned on $ \mathring{\eta}^{\delta}\circ\chi(s) $, it follows
\begin{align*}
\mathbb{E} \left|R_{i,2}\right| 
& \leq 4 \delta^{2} \|q\|_{\infty}^{2} \sum_{\substack{j,k=1\\j\neq k}}^{n}  \|\partial_{j} \partial_{k} q^+_{i}\|_{\infty} + 2 \delta\|q\|_{\infty} \sum_{j=1}^{n} \|\partial_{j}^{2} q^+_{i} \|_{\infty}.
\end{align*}
For each $i=1,\dots,n$, define
\[
\tilde{U}_{i}(s) \coloneqq \left(s-\chi(s) - \frac{\delta}{2}\right) \sum_{j=1}^{n} \partial_{j} q^+_{i}(\mathring{\eta}^{\delta}\circ\chi(s)) q_{j}(\mathring{\eta}^{\delta}\circ\chi(s)).
\]
Since $ \int^{\chi(t)+\delta}_{\chi(t)} |s- \chi(s) - \delta/2| \mathd s  = \delta^{2}/4 $ and $ \int^{\chi(t)+\delta}_{\chi(t)} (s- \chi(s) - \delta/2) \mathd s  = 0 $ for any bounded function $ g $ and any $ t > 0 $,
\[
\left| \int^{t}_{0} (1-\mathring{\eta}^{\delta}_{i} (s)) \left( s- \chi(s) - \frac{\delta}{2}\right) g(\chi(s))  \mathd s \right| \leq \delta^{2} \|g\|_{\infty} \left(1+\sum_{s\in[0,t]}\Delta\mathring{\eta}^{\delta}_i(s)\right).
\]
Again using the fact that the expected number of jumps in $ \mathring{\eta}^{\delta}_{i} $ on $[a,b] $ is bounded by $  2 \|q\|_{\infty} (b-a) $,
\begin{align}
\mathbb{E}\left| \int^{t}_{0} (1-\mathring{\eta}^{\delta}_{i}(s)) \tilde{U}_i(s) \mathd s \right| \leq 2 (t+1) \delta^{2} \|q\|_{\infty} (1 + \|q\|_{\infty} ) \sum_{j=1}^{n} \|\partial_{j} q^+_{i}\|_{\infty} \label{proof:mid:eq2b}.
\end{align}
Therefore, to control $T_+^{(2)}$, it now suffices to consider the remaining error term 
\[
\mathbb{E}\left| \int^{t}_{0} (1-\mathring{\eta}^{\delta}_{i}(s)) (U_{i}(s) - \tilde{U}_{i}(s)) \mathd s \right| \leq t \sup_{s \in [0,t]} \mathbb{E}|U_i(s) - \tilde{U}_i(s)|.
\]
To do so, let $M(t)$ denote the martingale
\begin{multline*}
M_{i}(t) =\mathring{\eta}^{\delta}_{i}(t) - \int^{t}_{0}((1-\mathring{\eta}^{\delta}_{i}(s))q^+_{i}(p^{\delta}\circ\mathring{\eta}^{\delta}\circ\chi(s)) \\- \mathring{\eta}^{\delta}_{i}(s) q^-_{i}(p^{\delta}\circ\mathring{\eta}^{\delta}\circ\chi(s))) \mathd s.
\end{multline*}
In terms of $M(t)$, the difference $U_i(s) - \tilde{U}_i(s) = T_{i,1}+T_{i,2}-T_{i,3}-T_{i,4}+R_{i,1}+R_{i,2}$, where
{\small\begin{align*}
T_{i,1} & = \sum_{j=1}^{n} \partial_{j} q^+_{i}(\mathring{\eta}^{\delta}\circ\chi(s)) (M_{j}(s) - M_{j}\circ\chi(s)) \\
T_{i,2} & = \sum_{j=1}^{n} \partial_{j} q^+_{i}(\mathring{\eta}^{\delta}\circ\chi(s)) \int^{s}_{\chi(s)}(1-\mathring{\eta}^{\delta}_{j}(u)) \left[q^+_{j}(p^{\delta}\circ\mathring{\eta}^{\delta}\circ\chi(u)) - q^+_{j}(\mathring{\eta}^{\delta}\circ\chi(u))\right] \mathd u \\
T_{i,3} & = \sum_{j=1}^{n} \partial_{j} q^+_{i}(\mathring{\eta}^{\delta}\circ\chi(s)) \int^{s}_{\chi(s)} \mathring{\eta}^{\delta}_{j}(u) \left[q^-_{j}(p^{\delta}\circ\mathring{\eta}^{\delta}\circ\chi(u)) - q^-_{j}(\mathring{\eta}^{\delta}\circ\chi(u))\right] \mathd u  \\
T_{i,4} & = \sum_{j=1}^n \partial_j q^+_i(\mathring{\eta}^{\delta}\circ\chi(s))[q^+_j(\mathring{\eta}^{\delta}\circ\chi(s))+q^-_j(\mathring{\eta}^{\delta}\circ\chi(s))]
\int_{\chi(s)}^s [\mathring{\eta}^{\delta}_j(u)-\mathring{\eta}^{\delta}_j\circ\chi(u)] \mathd u.  
\end{align*} }
First, to bound $T_{i,1}$, let $M^+_i(s) \coloneqq \sum_{j=1}^n \partial_j q^+_i(\mathring{\eta}^{\delta}\circ\chi(s))M_j(s)$, so that
\[
\mathbb{E}T_{i,1}^2 = \mathbb{E}[M^+_i(s)-M^+_i\circ \chi(s)]^2
= \mathbb{E}[M^+_i]_s - \mathbb{E}[M^+_i]_{\chi(s)}.
\]
Note that the quadratic variation of $ M_{j} $, (resp. $M^+_i$), is just the number of jumps in $M_{j} $ (resp. $M^+_i$) and that, recalling the defining property of a spin system,  with probability one, no two sites experience simultaneous jumps in $ \mathring{\eta}^{\delta} $. Therefore, $\mathbb{E}[M^+_i]_s - \mathbb{E}[M^+_i]_{\chi(s)} \leq 2 \delta \|q\|_{\infty} \sum_{j=1}^n \|\partial_j q^+_i\|_{\infty}^2$,
and so
\begin{equation}
(\mathbb{E}|T_{i,1}|)^2 \leq 2\delta\|q\|_{\infty}  \sum_{j=1}^{n} \| \partial_{j} q^+_{i} \|_{\infty}^{2}. \label{proof:mid:eq4}
\end{equation}
The terms $T_{i,2}$ and $T_{i,3}$ are controlled together: by the Mean Value Theorem, 
\begin{align*}
|q^+_j(p^{\delta} \circ \mathring{\eta}^{\delta} \circ \chi(s)) - q^+_j(\mathring{\eta}^{\delta}\circ \chi(s))|
&\leq \frac12 \delta \|q\|_{\infty} \sum_{k=1}^{n} \|\partial_{k} q^+_{j}\|_{\infty}, 
\end{align*}
and similarly for $q_j^-$, which implies
\begin{align}
\mathbb{E}|T_{i,2}| + \mathbb{E}|T_{i,3}| & \leq \frac12 \delta^{2} \|q\|_{\infty} \sum_{j,k=1}^{n} \|\partial_{j} q^+_{i}\|_{\infty} \left( \|\partial_{k} q^+_{j}\|_{\infty} + \|\partial_{k} q^-_{j}\|_{\infty} \right). \label{proof:mid:eq5}
\end{align}
Finally, to bound $T_{i,4}$,
\begin{align}
\mathbb{E}|T_{i,4}| & \leq 2\|q\|_{\infty} \sum_{j=1}^{n} \|\partial_{j} q^+_{i}\|_{\infty}  \int^{s}_{\chi(s)} \mathbb{E}| \mathring{\eta}^{\delta}_{j}(u) - \mathring{\eta}^{\delta}_{j}\circ\chi(u) | \mathd u \nonumber \\
& \leq 2\delta \|q\|_{\infty} \sum_{j=1}^{n} \|\partial_{j} q^+_{i}\|_{\infty} \  \mathbb{E} \sum_{u \in [\chi(s),s]} \Delta \mathring{\eta}^{\delta}_j(u) \leq 4 \delta^{2} \|q\|^{2}_{\infty} \sum_{j=1}^{n} \|\partial_{j} q^+_{i}\|_{\infty}.\label{proof:mid:eq6}
\end{align}
Combining (\ref{proof:mid:eq2b})-(\ref{proof:mid:eq6}) with the bounds on $ R_{i,1} $ and $ R_{i,2} $ gives
\begin{multline*}
\mathbb{E}|T_+^{(2)}|
\leq 5(t+1)n \delta^2\|q\|_{\infty}(1+\|q\|_{\infty})[\|D^{\ast}q\|_1 + \Gamma_n + \|D^{\ast}q\|_1^2] \\+ 2 \delta t \|q\|_{\infty} \gamma_n + \sqrt{2} \delta^{1/2} \|q\|_{\infty}^{1/2} t  \|D^{\ast} q\|_{2,1}
\end{multline*}
and so $\mathbb{E}|T_+^{(2)}| \leq 5 \alpha(n,\delta) (1 + t)$.
Applying the same arguments to $ T^{(2)}_{-} $ yields 
\begin{multline*}
\sum_{i=1}^n \mathbb{E}|\eta_i(t) - \mathring{\eta}^{\delta}_i(t)| \leq 2(\|q\|_{\infty} + \|D^{\ast}q\|_1)\int_0^t \sum_{i=1}^n \mathbb{E}|\eta_i(s)-\mathring{\eta}^{\delta}_i(s)|\mathd s
\\+ 10 \alpha(n,\delta)(1 + t).
\end{multline*}
The proof is completed by applying Gronwall's inequality as in the proof of Theorem \ref{thm:EulerRate}.
\end{proof}

\subsection{Proof of Theorem \ref{thm:EulerExact}}
\label{sec:EulerExactProof}
Throughout this proof, for the sake of brevity, the notation
 $x \lesssim y$ will be used to imply there exists
some universal constant $c > 0$ such that $x \leq c y$. Naturally, we assume throughout that the hypotheses of Theorem \ref{thm:EulerExact} hold. 
We begin by coupling $ \eta(t) $ and $ \eta^\delta(t) $ using the basic coupling, that is, for independent unit-rate Poisson processes $N_{i,k}^{+}, N_{i,k}^{-}$
with $i=1,\dots,n;\ k=1,2,3$:
\begin{align*}
\eta_{i}(t) 
 & = \eta_{i}(0) + N_{i,1}^{+} \left(\int^{t}_{0}  [(1-\eta_{i}(s)) q^+_{i}(\eta(s))  \wedge (1-\eta^\delta_{i}(s) q^+_{i}(\eta^\delta \circ \chi(s))]  \mathd s \right) \\
& + N_{i,2}^{+} \left(\int^{t}_{0} [ (1-\eta_{i}(s))q^+_{i}(\eta(s))  -  (1-\eta^\delta_{i}(s)) q^+_{i}(\eta^\delta \circ \chi(s)) ]_{+}  \mathd s \right) \\
& - N_{i,1}^{-}\left(\int^{t}_{0} [\eta_{i}(s)q^-_{i}(\eta(s)) \wedge \eta^\delta_{i}(s) q^-_{i}(\eta^\delta\circ\chi(s))] \mathd s \right)  \\
& - N_{i,2}^{-}\left(\int^{t}_{0} [\eta_{i}(s)q^-_{i}(\eta(s)) - \eta^\delta_{i}(s) q^-_{i}(\eta^\delta\circ\chi(s))] _{+} \mathd s \right)
\end{align*}
and
\begin{align*}
\eta^\delta_{i}(t) 
 & = \eta_{i}(0) + N_{i,1}^{+} \left(\int^{t}_{0}  \left[(1-\eta_{i}(s)) q^+_{i}(\eta(s))  \wedge (1-\eta^\delta_{i}(s) q^+_{i}(\eta^\delta \circ \chi(s))\right]  \mathd s \right) \\
& + N_{i,3}^{+} \left(\int^{t}_{0} [ (1-\eta^\delta_{i}(s)) q^+_{i}(\eta^\delta \circ \chi(s)) - (1-\eta_{i}(s))q^+_{i}(\eta(s))]_{+}  \mathd s \right) \\
& - N_{i,1}^{-}\left(\int^{t}_{0} [\eta_{i}(s)q^-_{i}(\eta(s)) \wedge \eta^\delta_{i}(s) q^-_{i}(\eta^\delta\circ\chi(s))] \mathd s \right)  \\
& - N_{i,3}^{-}\left(\int^{t}_{0} [ \eta^\delta_{i}(s) q^-_{i}(\eta^\delta\circ\chi(s)) - \eta_{i}(s)q^-_{i}(\eta(s))] _{+} \mathd s \right).
\end{align*}
Centering the Poisson processes, there is
\begin{multline}
\label{eq:MartDefn}
\eta_i(t) - \eta^\delta_i(t) = M_i(t) + \int_0^t[(1-\eta_i(s))q^+_i(\eta(s)) - (1-\eta^\delta_i(s))q^+_i(\eta^\delta\circ\chi(s))] \mathd s\\
- \int_0^t[\eta_i(s)q^-_i(\eta(s)) - \eta^\delta_i(s)q^-_i(\eta^\delta\circ\chi(s))] \mathd s.
\end{multline}
where the remainder $M_i(t)$ is a martingale. 
For $\phi \in \Phi$, let $\phi_i = \phi(z_i)$ for each $i=1,\dots,n$. For the sake of brevity, let $M_\delta^\phi(t) \coloneqq (n\delta)^{-1} \sum_{i=1}^n \phi_i M_i(t)$. Then from (\ref{eq:MartDefn}) and (\ref{eq:mathcalE}), \begin{multline}
\label{EA:main}
(n\delta)^{-1} \sum_{i=1}^{n} \phi_{i} ( \eta_i(t) - \eta^\delta_{i}(t) - \delta \mathcal{E}_{i}(t) ) = M^{\phi}_\delta(t) \\
-\int_0^t n^{-1}\sum_{i=1}^n \phi_i[\delta^{-1}(\eta_i(s)-\eta^\delta_i(s))-\mathcal{E}_i(s)][q^+_i(\eta^\delta(s))+q^-_i(\eta^\delta(s))] \mathd s \\
+S_+(\phi,t) - S_-(\phi,t) + \sum_{k=1}^3 R_+^{(k)}(\phi,t) + \sum_{k=1}^3 R_-^{(k)}(\phi,t) + D_+(\phi,t) + D_-(\phi,t),
\end{multline}
where $S_+$ and $S_-$ denote the second-order weak error terms
{\small\begin{align*}
S_+(\phi,t) &= \int_{0}^{t}n^{-1}\sum_{i=1}^{n}\phi_{i}(1-\eta_{i}(s))\left[\frac{q^+_{i}(\eta(s))-q^+_{i}(\eta^\delta(s))}{\delta}-\sum_{j=1}^{n}\partial_{j}q^+_{i}(\rho(s))\mathcal{E}_{j}(s)\right] \mathd s \\
S_-(\phi,t) &= \int_{0}^{t}n^{-1}\sum_{i=1}^{n}\phi_{i}\eta_{i}(s)\left[\frac{q^-_{i}(\eta(s))-q^-_{i}(\eta^\delta(s))}{\delta}-\sum_{j=1}^{n}\partial_{j}q^-_{i}(\rho(s))\mathcal{E}_{j}(s)\right] \mathd s,
\end{align*}}
$D_\pm$ denote the discretisation error terms
\begin{align*}
D_{+}(\phi,t) &= \int_{0}^{t}n^{-1}\sum_{i=1}^{n}\phi_{i}(1-\eta^\delta_{i}(s))\frac{q^+_{i}(\eta(s))-q^+_{i}(\eta^\delta\circ\chi(s))}{\delta} \mathd s \\
&\qquad\qquad-\frac{1}{2}\int_{0}^{t}n^{-1}\sum_{i=1}^{n}\phi_{i}(1-\rho_{i}(s))\sum_{j=1}^{n}\partial_{j}q^+_{i}(\rho(s))q_{j}(\rho(s)) \mathd s \\
D_{-}(\phi,t) &= \int_{0}^{t}n^{-1}\sum_{i=1}^{n}\phi_{i}\eta^\delta_{i}(s)\frac{q^-_{i}(\eta^\delta(s))-q^-_{i}(\eta^\delta\circ\chi(s))}{\delta} \mathd s \\
&\qquad\qquad-\frac{1}{2}\int_{0}^{t}n^{-1}\sum_{i=1}^{n}\phi_{i}\rho_{i}(s)\sum_{j=1}^{n}\partial_{j}q^-_{i}(\rho(s))q_{j}(\rho(s))\mathd s,
\end{align*}
and $R_{\pm}^{(k)}$ denote the remainder terms:
\begin{align*}
R_{\pm}^{(1)}(\phi,t) &= \int_0^t n^{-1}\sum_{i=1}^n \phi_i (\rho_i(s)-\eta_i(s))\sum_{j=1}^n \partial_j q^{\pm}_i(\rho(s)) \mathcal{E}_j(s) \mathd s \\
R_{\pm}^{(2)}(\phi,t) &= \int_0^t n^{-1}\sum_{i=1}^n \phi_i \mathcal{E}_i(s)[q^{\pm}_i(\rho(s))-q^{\pm}_i(\eta(s))] \mathd s \\
R_{\pm}^{(3)}(\phi,t) &= \int_0^t n^{-1}\sum_{i=1}^n \phi_i \mathcal{E}_i(s)[q^{\pm}_i(\eta(s))-q^{\pm}_i(\eta^{\delta}(s))]\mathd s.
\end{align*}
Observe that, under Assumption \ref{ass:Main}, the second term on the right-hand side of (\ref{EA:main}) can be bounded in magnitude by
\[
2\|q\|_{\Phi} \int^{T}_{0} \sup_{s\in[0,t]} \sup_{\phi\in\Phi} \left| (n\delta)^{-1}\sum_{i=1}^{n} \phi_{i}[ \eta_{i}(s) - \eta^\delta_{i}(s) - \delta\mathcal{E}_{i}(s)] \right| \mathd t,
\]
suggesting that we may use Gronwall's inequality on (\ref{EA:main})
to obtain the desired bound. The remainder of the proof comprises
controlling the other terms, in order of their appearance.

\subsubsection{Controlling the martingale term}
Fixing $T > 0$, let $\mathcal{M}_\delta(\phi) = \sup_{t\in[0,T]} | M_\delta^{\phi}(t) |$. Due to (\ref{eq:MartDefn}), 
\begin{equation}
\label{eq:MartAbs}
\mathcal{M}_\delta(\phi) \leq (n\delta)^{-1} \sup_{t\in[0,T]}\sum_{i=1}^n \phi_i |M_i(t)| \leq \delta^{-1} \|\phi\|_{\infty}(1 + 2T \|q\|_{\infty}).
\end{equation}
Central to our strategy of controlling the martingale $M_\delta^\phi$ is the Fuk-Nagaev inequality of \cite{DvZ:01}.
First observe that the predictable quadratic variation of $M_\delta^\phi$ is
\begin{multline*}
\langle M_\delta^\phi \rangle_t = \int_0^t \sum_{i=1}^n (n\delta)^{-2} \phi_i^2 \left|(1-\eta_i(s))q^+_i(\eta(s)) - (1 - \eta^\delta_i(s))q_i^+(\eta^\delta\circ\chi(s))\right| \mathd s \\
+ \int_0^t \sum_{i=1}^n (n\delta)^{-2} \phi_i^2 \left|\eta_i(s)q_i^-(\eta(s)) - \eta^\delta_i(s)q_i^-(\eta^\delta\circ\chi(s))\right| \mathd s.
\end{multline*}
Adopting the notation from the proof of Theorem \ref{thm:EulerRate}, for any $\Phi' \subset \Phi$, 
\[
\sup_{\phi \in \Phi'} \langle M_\delta^{\phi} \rangle_t \leq \left(\frac{\|\Phi'\|_{\infty}}{n\delta}\right)^2
[T_+^{(1)} + T_-^{(1)}+T_+^{(2)}+T_-^{(2)}],
\]
and so by following the same arguments,
\[
\mathbb{E}\sup_{\phi \in \Phi'} \sup_{t\in[0,T]}\langle M_\delta^{\phi} \rangle_t \leq \frac{4 T \|\Phi'\|_{\infty}^2}{n\delta} 
\|q\|_{\infty}\|D^{\ast}q\|_1 e^{2T(\|q\|_{\infty}+\|D^{\ast}q\|_1)} \eqqcolon \frac{C_T \|\Phi'\|_{\infty}^2}{n \delta}.
\]
\begin{lemma}
\label{lem:MartingaleExactBound}
Assuming that $\delta \geq n^{-1}$, there exists a constant $C > 0$ depending only on $\Omega,\Omega',d$ such that for
any $T > 0$, 
\[
\bar{\mathbb{E}} \sup_{\phi \in \Phi} \sup_{t\in[0,T]} M_\delta^{\phi}(t) \lesssim \frac{1+\log(n\delta)}{(n\delta)^{1/3}}
\left[ 1 + 2T \|q\|_{\infty} \|D^{\ast}q\|_{1} e^{2T(\|q\|_{\infty}+\|D^{\ast}q\|_1)} \right].
\]
\end{lemma}
\begin{proof}
For any $\Phi' \subset \Phi$, let $\mathcal{M}_\delta(\Phi') \coloneqq \sup_{\phi \in \Phi'} \mathcal{M}_\delta(\phi)$. Assuming
for now that $\Phi'$ is a finite subset of $\Phi$, by the Fuk-Nageav inequality \cite[Corollary 3.4]{DvZ:01}, for any $x, L > 0$,
\begin{equation}
\label{eq:MartFiniteProb}
\mathbb{P}\left(\mathcal{M}_\delta(\Phi') \geq \sqrt{2 L x} + \frac{2 \|\Phi'\|_{\infty}}{3n\delta} x \right)
\leq 2|\Phi'|e^{-x} + \frac{C_T \|\Phi'\|_{\infty}^2}{n\delta L}.
\end{equation}
To apply the inequality in (\ref{eq:KyFanBounds}), it suffices to choose $x$ and $L$ such that $\sqrt{2Lx} = (n\delta L)^{-1} C_T \|\Phi'\|_{\infty}^{2}$ and $(3n\delta)^{-1}\|\Phi'\|_{\infty} x = |\Phi'|e^{-x}$. The only real solution to these relations is given by
\begin{equation}
\label{eq:MartULSoln}
x = W_0\left(\frac{3 |\Phi'| n\delta}{\|\Phi'\|_{\infty}}\right),\qquad \qquad
L = \left(\frac{C_T\|\Phi'\|_{\infty}^{2}}{n\delta}\right)^{2/3} \cdot \frac{1}{(2x)^{1/3}},
\end{equation}
where $W_0$ is the principal branch cut of the Lambert $W$-function. Now, since
$W_0(x) \leq \log(1+x)$ for all $x > 0$,
(\ref{eq:KyFanBounds}), (\ref{eq:MartFiniteProb}), and (\ref{eq:MartULSoln}) imply that 
\begin{multline*}
\bar{\mathbb{E}}\mathcal{M}_\delta(\Phi') \leq \frac{4\|\Phi'\|_{\infty}}{3n\delta}\cdot \log\left(1 + \frac{3 |\Phi'| n\delta}{\|\Phi'\|_{\infty}}\right) \\ + 2\left(\frac{2 C_T}{n\delta}\right)^{1/3}\|\Phi'\|_{\infty}^{2/3}\cdot
\log\left(1 + \frac{3 |\Phi'| n\delta}{\|\Phi'\|_{\infty}}\right)^{1/3}.
\end{multline*}
Assuming $ n\delta  \geq 1 $, 
$
\log(1+\frac{3|\Phi'|n\delta}{\|\Phi'\|_{\infty}})\leq 2[1+\log(n\delta)+\log(\frac{|\Phi'|}{\|\Phi'\|_{\infty}})],
$
so
\begin{multline}
\label{eq:MartFinite}
\bar{\mathbb{E}}\mathcal{M}_\delta(\Phi') \leq 6 (n\delta)^{-1/3} [\|\Phi'\|_{\infty} + (1 +  C_T) \|\Phi'\|_{\infty}^{2/3}] \\ \left[1 + \log(n\delta) + \log\left(\frac{|\Phi'|}{\|\Phi'\|_{\infty}}\right)\right].
\end{multline}
The extension of the estimate (\ref{eq:MartFinite}) to $\Phi' = \Phi$ relies on the
chaining argument of Dudley. Let $\{\phi^{(k)}\}_{k=1}^\infty$ be a sequence of
random elements in $\Phi$ adapted to the same probability space as $M(t)$, for which
$\mathcal{M}_\delta(\phi^{(k)}) \to \mathcal{M}_\delta(\Phi)$ in the
topology of pointwise convergence, and $\|\phi^{(k)} - \phi^{(l)}\|_{\infty} \leq 2^{-m}$ for $k,l \geq m$. Such a sequence is guaranteed to exist since $\Phi$ is relatively compact in the space of continuous functions on $\Omega$. From this, construct a new sequence $\{\bar{\phi}^{(k)}\}_{k=1}^{\infty}$ such that $\bar{\phi}^{(k)}$ occupies the $ 2^{-k}$-internal covering of $ \Phi $, $\Phi_{2^{-k}}$, and $\|\bar{\phi}^{(k)} - \phi^{(k)}\|_{\infty} \leq 2^{-k}$ for each
$k=1,2,\dots$. Applying dominated convergence under the estimate (\ref{eq:MartAbs}),
\[
\bar{\mathbb{E}} \mathcal{M}_\delta(\Phi) \leq \sum_{k=1}^{\infty} \bar{\mathbb{E}} \mathcal{M}_\delta(\bar{\phi}^{(k)} - \bar{\phi}^{(k-1)}).
\]
Observe that for $k=1,2,\dots$, $\|\bar{\phi}^{(k)} - \bar{\phi}^{(k-1)}\|_{\infty} \leq 5 \cdot 2^{-k}$. 
Now, letting
\[
\Phi^{(k)} = \{z = x - y\,:\, \ x \in \Phi_{2^{-k}}, y \in \Phi_{2^{-(k-1)}}, \|z\|_{\infty} \leq 5 \cdot 2^{-k}\}, 
\]
evidently, $\|\Phi^{(k)}\|_{\infty} \leq 5 \cdot 2^{-k}$ and $|\Phi^{(k)}| \leq |\Phi_{2^{-k}}| \cdot |\Phi_{2^{-(k-1)}}|$ and so by \cite[Theorem 9.2]{vitushkin},
there is a constant $C > 0$ depending only on $\Omega,\Omega',d$ such that $\log |\Phi^{(k)}| \leq C k^{d+1}$. Since $\bar{\phi}^{(k)} - \bar{\phi}^{(k-1)} \in \Phi^{(k)}$ by construction, $\bar{\mathbb{E}} \mathcal{M}_\delta(\Phi) \leq \sum_{k=1}^{\infty} \bar{\mathbb{E}} \mathcal{M}_\delta(\Phi^{(k)})$ and so (\ref{eq:MartFinite}) implies
\[
\bar{\mathbb{E}} \mathcal{M}_\delta(\Phi) \lesssim \frac{1+C_T}{(n\delta)^{1/3}} \sum_{k=1}^{\infty} \frac{1 + \log(n\delta) + k^{d+1}}{4^{k/3}},
\]
where $C > 0$ depends only on $\Omega, \Omega', d$. The series is absolutely
convergent, implying the result. 
\end{proof}

\subsubsection{The second-order error terms} 

Moving on to the second-order error terms $S_+$ and $S_-$, observe that
\begin{multline}
\label{eq:SEDecomp}
\sup_{t\in[0,T]}\sup_{\phi\in\Phi}|S_+(\phi,t)| \\
\leq \|Dq\|_{\Phi} \int_0^T \sup_{s\in[0,t]}\sup_{\phi\in\Phi}
\left|n^{-1}\sum_{j=1}^n \phi_j [\delta^{-1}(\eta_j(s)-\eta^\delta_j(s))-\mathcal{E}_j(s)]\right|\mathd t\\
+ \int_0^T (n\delta)^{-1}\sum_{i=1}^n \left|q_i^+(\eta(s))-q_i^+(\eta^\delta(s))-\sum_{j=1}^n \partial_j q_i^+(\rho(s))(\eta_j(s)-\eta^\delta_j(s))\right|\mathd s,
\end{multline}
and similarly for $S_-$. The first term on the right-hand side of (\ref{eq:SEDecomp})
will contribute to our use of Gronwall's inequality on (\ref{EA:main}), so we need only consider the second term. This is treated for $S_+$ in Lemma \ref{lem:SecOrderTerm}; the procedure for $S_-$ is identical. 
For brevity, we leave the inequality (\ref{Lem:Term2:Eq1}) in terms of the $L^1$ and $L^2$ norms of $\bar{J}(T)$ and $\bar{J}_\delta(T)$, recalling that these quantities are bounded according to Lemmas \ref{lem:JWBarBound} and \ref{lem:JVBarBound}. 
\begin{lemma} \label{lem:SecOrderTerm}
For any $T > 0$, there is
\begin{multline}
\mathbb{E}\int^{T}_{0} (n\delta)^{-1}  \sum_{i=1}^{n}  \left| q_i^+(\eta(s)) - q_i^+(\eta^\delta(s)) - \sum_{j=1}^{n} \partial_{j} q_i^+(\rho(s)) (\eta_{j}(s) - \eta^\delta_j(s))  \right| \mathd s \label{Lem:Term2:Eq1} \\ 
\lesssim 
n \tilde{\Gamma}_n T^{2} \|q\|_{\infty} \|D^{\ast}q\|_{1} e^{4T(\|q\|_{\infty} + \|D^{\ast}q\|_{1})} (\mathbb{E}\bar{J}(T) + n^{-1/2} + \delta T \|q\|_{\infty} \|D^{\ast}q\|_{1}) \\
+ 
\frac{T \sqrt{1+ \log n} }{n \delta} ( n \tilde{\Gamma}_n +  n^{1/2} \theta_n )  (1 + n(\|\bar{J}(T)\|_{2}^{2} + \|\bar{J}_{\delta}(T)\|_{2}^{2}))
.
\end{multline}
\end{lemma}
\begin{proof}
The proof proceeds in a similar fashion to \cite[Proposition 9]{HMP:18}, comparing $\eta(t)$ to $\rho(t)$, $\rho(t)$ to $\rho^{\delta}(t)$, and finally $\rho^{\delta}(t)$ to $\eta^\delta(t)$. First, from Lemma \ref{lem:JWBarBoundSecond},
\begin{multline}
\sum_{i=1}^n \mathbb{E} \left| q_i^+(\eta(s))  - q_i^+(\rho(s)) - \sum_{j=1}^{n} \partial_{j} q_i^+(\rho(s)) ( \eta_{j}(s)   - \rho_{j}(s)) \right| \\
+\sum_{i=1}^n\mathbb{E}\left| q_i^+(\eta^\delta(s))-q_i^+(\rho^\delta(s))-\sum_{j=1}^n \partial_j q_i^+(\rho^\delta(s)) (\eta^\delta_j(s) - \rho^\delta(s))\right| \\
\lesssim  \sqrt{1+ \log n} ( n \tilde{\Gamma}_n +  n^{1/2}\theta_n) (1 + n\|\bar{J}(s)\|_{2}^{2} + n \|\bar{J}_\delta(s)\|_2^2).\label{eq:SecondOrderMainTerm1}
\end{multline}
Following the arguments of the proof of Theorem \ref{thm:EulerRate}, 
\begin{equation}
\label{eq:DeterCompare}
\sum_{j=1}^{n} |\rho_j(t) - \rho^\delta_j(t)| \leq 4n\delta t \|q\|_{\infty} \|D^{\ast}q\|_{1} e^{2t(\|q\|_{\infty} + \|D^{\ast}q\|_{1})},
\end{equation}
and so a second order Taylor expansion for $ q_i^+ $ implies
\begin{multline}
\label{eq:SecondOrderMainTerm3}
\sum_{i=1}^n \left| q_i^+(\rho^\delta(s))   - q_i^+(\rho(s))  - \sum_{j=1}^{n} \partial_{j} q_i^+(\rho(s)) (\rho^\delta_j(s) - \rho_{j}(s)) \right| \\
\lesssim \Gamma_n \left(n\delta s \|q\|_{\infty} \|D^{\ast}q\|_{1} e^{2s(\|q\|_{\infty} + \|D^{\ast}q\|_{1})} \right)^{2}   .
\end{multline}
For the remaining cross terms, once again due to (\ref{eq:DeterCompare}),
\begin{multline}
\label{eq:SecondOrderCrossTerm1}
\sum_{i=1}^n \mathbb{E} \left|\sum_{j=1}^{n} (\eta^\delta_j(s) - \omega^\delta_{j}(s)) [ \partial_{j} q_i^+(\rho^\delta(s)) - \partial_{j} q_i^+(\rho(s))] \right| \\
\leq 4n \delta s \|q\|_{\infty} \|D^{\ast}q\|_{1} e^{2s(\|q\|_{\infty} + \|D^{\ast}q\|_{1})} \tilde{\Gamma}_n \|\bar{J}_\delta(s) \|_{1},
\end{multline}
and since the $ \omega^\delta_{j}(t) $ are independent with $ \mathbb{P}(\omega^\delta_{j}(t)=1) = \rho_j^\delta(t) $,
\begin{align}
\MoveEqLeft \sum_{i=1}^n\mathbb{E} \left| \sum_{j=1}^{n} [ \omega^\delta_{j}(s) - \rho_j^\delta(s) ] [ \partial_{j} q_i^+(\rho^\delta(s)) - \partial_{j} q_i^+(\rho(s)) ] \right| \nonumber \\
& \leq \left( \sum_{j=1}^{n} \rho^\delta_j(s) (1-\rho^\delta_j(s)) \left[ \partial_{j} q_i^+(\rho^\delta(s)) - \partial_{j} q_i^+(\rho(s))\right]^{2} \right)^{1/2} \nonumber \\
& \leq 2n^{1/2} \delta s \|q\|_{\infty} \|D^{\ast}q\|_{1} e^{2s(\|q\|_{\infty} + \|D^{\ast}q\|_{1})} \tilde{\Gamma}_n. \label{eq:SecondOrderCrossTerm2}
\end{align}
Combining (\ref{eq:SecondOrderMainTerm1}), 
 (\ref{eq:SecondOrderMainTerm3}), (\ref{eq:SecondOrderCrossTerm1}), and (\ref{eq:SecondOrderCrossTerm2}) implies (\ref{Lem:Term2:Eq1}). 
\end{proof}

\subsubsection{The remainder terms $R_{\pm}^{(k)}$}
Next, we turn our attention to the remainder terms $R_{\pm}^{(k)}$ for $k=1,2,3$. The first of these terms, $R_{\pm}^{(1)}$, is controlled in
Lemma \ref{lem:Remainder1}. 
\begin{lemma}
\label{lem:Remainder1}
For any $ T> 0 $,
\begin{multline*}
\mathbb{E}\sup_{t\in[0,T]}\sup_{\phi\in\Phi} (|R_{+}^{(1)}(\phi,t)| + |R_{-}^{(1)}(\phi,t)|) \\
\lesssim \|D^{\ast}q\|_{\infty} \left( \int^{T}_{0}  \| \mathcal{E}(s)\|_{\infty} \mathd s \right)( \mathbb{E} \bar{J}(T)  + n^{-1/2})
\end{multline*}
\end{lemma}
\begin{proof}
It will suffice to consider $R_{+}^{(1)}$, as the process for $R_{-}^{(1)}$ is identical. Using the independent site approximation $ W(t) $,
\begin{multline}
\label{eq:Remainder1Decomp}
\sup_{t\in[0,T]}\sup_{\phi\in\Phi}|R_{+}^{(1)}(\phi,t)| \\\leq \int^{T}_{0} \sup_{\phi \in \Phi} \left| n^{-1}\sum_{i=1}^{n} \phi_{i} \left[ \sum_{j=1}^{n} \partial_{j} q^+_{i}(\rho(s)) \mathcal{E}_{j}(s) \right] (\rho_{i}(s)  -\omega_{i}(s)) \right|  \mathd s \\
+ \int^{T}_{0} \sup_{\phi \in \Phi} \left| n^{-1}\sum_{i=1}^{n} \phi_{i} \left[ \sum_{j=1}^{n} \partial_{j} q^+_{i}(\rho(s)) \mathcal{E}_{j}(s) \right] (\omega_{i}(s)  - \eta_{i}(s)) \right|  \mathd s 
\end{multline}
The second term on the right-hand side of (\ref{eq:Remainder1Decomp})
is bounded by \preprintlb $\|Dq^+\|_{\infty} \bar{J}(T)\int_0^T \|\mathcal{E}(s)\|_{\infty} \mathd s$. Turning our attention now to the first term,
let $ \tilde{\omega}(s) $ be an independent copy of $ \omega(s) $. Using the symmetrisation method (see \cite[\S3]{Devroye:2001aa}), 
\begin{multline*}
\mathbb{E} \sup_{\phi \in \Phi} \left| n^{-1}\sum_{i=1}^{n} \phi_{i} \left[ \sum_{j=1}^{n} \partial_{j} q^+_{i}(\rho(s)) \mathcal{E}_{j}(s) \right] (\rho_{i}(s)  -\omega_{i}(s)) \right| \\
\leq \mathbb{E} \sup_{\phi \in \Phi} \left| n^{-1}\sum_{i=1}^{n} \phi_{i} \left[ \sum_{j=1}^{n} \partial_{j} q^+_{i}(\rho(s)) \mathcal{E}_{j}(s) \right] (\tilde{\omega}_{i}(s)  -\omega_{i}(s)) \right| \leq \Rad (\Psi(s; \bv z)),
\end{multline*}
where $ \Psi(s; \bv z) = \{\psi: \psi_{i} = \phi_{i} \left[ \sum_{j=1}^{n} \partial_{j} q^+_{i}(\rho(s)) \mathcal{E}_{j}(s) \right] ,\ \phi \in \Phi \} $. Finally, applying \cite[Lemma 26.9]{SSBD2014} with (\ref{eq:RadBound}) gives 
\[
\Rad(\Psi(s; \bv z)) \leq \|Dq^+\|_{\infty} \|\mathcal{E}(s)\|_{\infty}\Rad(\Phi(\bv z)) \lesssim n^{-1/2} \|Dq^+\|_{\infty}\|\mathcal{E}(s)\|_{\infty},
\]
which proves the lemma.
\end{proof}


The second terms, $R_{\pm}^{(2)}$, can be bounded almost directly using Lemma
\ref{lem:JWBarBoundCor}. Indeed, for any $T > 0$,
\begin{multline}
\label{eq:Remainder2}
\mathbb{E}\sup_{t\in[0,T]}\sup_{\phi\in\Phi} |R_{+}^{(2)}(\phi,t)|
\leq \int_0^T n^{-1}\sum_{i=1}^n |\mathcal{E}_i(s)|\mathbb{E}|q^+_i(\rho(s))-q^+_i(\eta(s))| \mathd s \\
\leq n^{-1} \left( \int^{T}_{0}   \|\mathcal{E}(s)\|_{\infty} \mathd s \right) [\|D q^+\|_{1}  (\|D^{\ast}q\|_{2,1}+\gamma_n)Te^{2T\|D^{\ast}q\|_{1}} + \\\tfrac{1}{2}\|Dq^+\|_{2,1} +\tfrac{1}{2}\gamma_n].
\end{multline}
A similar result also holds for $R_{-}^{(2)}(\phi,t)$. Therefore, it only remains
to bound $R_{\pm}^{(3)}$, which is accomplished in Lemma \ref{lem:Remainder3}.
\begin{lemma}
\label{lem:Remainder3}
For any $ T > 0 $,
\begin{multline*}
\mathbb{E}\sup_{t\in[0,T]}\sup_{\phi\in\Phi}(|R_{+}^{(3)}(\phi,t)|+|R_{-}^{(3)}(\phi,t)|)\\\leq 8 \delta T \|q\|_{\infty} \|D^{\ast}q\|_1^2 e^{2T(\|q\|_{\infty}+\|D^{\ast}q\|_1)} \int^{T}_{0}  \|\mathcal{E}(s)\|_{\infty} \mathd s.
\end{multline*}
\end{lemma}
\begin{proof}
The result follows by an application of Theorem \ref{thm:EulerRate}, upon
observing that
\begin{align*}
\sup_{t\in[0,T]} \sup_{\phi\in\Phi} |R_{+}^{(3)}(\phi,t)|
& \leq  \int^{T}_{0}  \|\mathcal{E}(s)\|_{\infty}  n^{-1} \sum_{i=1}^{n}  | q^+_{i}(\eta(s)) - q^+_{i}(\eta^{\delta}(s))|  \mathd s  \\
& \leq  \|Dq^+\|_1 \int^{T}_{0}  \|\mathcal{E}(s)\|_{\infty} n^{-1}  \sum_{j=1}^{n}  | \eta_{j}(s) - \eta^{\delta}_{j}(s)|  \mathd s.  
\end{align*}
Repeating this process for $R_{-}^{(3)}$ completes the proof.
\end{proof}

\subsubsection{The discretisation terms $D_+$ and $D_-$} 
Easily the most involved
portion of our proof of Theorem \ref{thm:EulerExact} involves bounding the discretisation terms
$D_+$ and $D_-$. We demonstrate the process for $D_+$ only, as it is
virtually identical for $D_-$. The first step is to compare $ q_i^+(\eta^\delta(s)) - q_i^+(\eta^\delta\circ \chi(s)) $ to $ \sum_{j=1}^{n} \partial_{j} q_i^+(\eta^\delta \circ \chi(s)) (\eta^\delta_j(s) -\eta^\delta_j\circ\chi(s))$ by
$\bar{\mathbb{E}}\sup_{t\in[0,T]}\sup_{\phi\in\Phi}|D_+(\phi,t)| \leq D_+^{(1)} + D_+^{(2)}$, where
\begin{multline*}
D_+^{(1)} = \mathbb{E}\int^{T}_{0} (n\delta)^{-1} \sum_{i=1}^{n} \left| q_i^+(\eta^\delta(s)) - q_i^+(\eta^\delta\circ \chi(s)) - \right. \\
\left. \sum_{j=1}^{n} \partial_{j} q_i^+(\eta^\delta \circ \chi(s)) (\eta^\delta_j(s) -\eta^\delta_j\circ\chi(s))  \right| \mathd s,
\end{multline*}
\vspace{-.5cm}
\begin{multline*}
D_+^{(2)} = \sup_{t\in[0,T]} \sup_{\phi \in \Phi} \mathbb{E}\left| 
\frac{1}{2} \int^{t}_{0}  n^{-1}\sum_{i,j=1}^{n} \phi_{i} (1-\rho_{i}(s)) \partial_{j} q_i^+(\rho(s)) q_{j}(\rho(s)) \mathd s \right.\\
\left. - \int^{t}_{0}  (n\delta)^{-1}\sum_{i,j=1}^{n} \phi_{i} (1-\eta^\delta_i(s)) \partial_j q_i^+(\eta^\delta \circ \chi(s)) (\eta^\delta_j(s) -\eta^\delta_j\circ\chi(s)) \mathd s \right|. 
\end{multline*}
The error in $D_+^{(1)}$ is bounded in the following Lemma \ref{LemLB:Eq1}.
\begin{lemma}
\label{LemLB:Eq1}
For any $T > 0$, $D_+^{(1)} \lesssim  T ( n^{-1}\|q\|_{\infty} + 2\delta \|q\|_{\infty}^{2}) n \tilde{\Gamma}_n$.
\end{lemma}

\begin{proof}
A second-order Taylor expansion yields
\begin{multline*}
\sum_{i=1}^n\mathbb{E} \left| q_i^+(\eta^\delta(s)) - q_i^+(\eta^\delta \circ \chi(s)) - \sum_{j=1}^{n} \partial_j q_i^+(\eta^\delta \circ \chi(s)) (\eta^\delta_j(s) -\eta^\delta_j\circ\chi(s)) \right| \\
\leq \frac{1}{2} \sum_{i,j,k=1}^{n} \| \partial_{j} \partial_{k} q_i^+ \|_{\infty} \mathbb{E} \left|(\eta^\delta_j(s) - \eta^\delta_j \circ \chi(s)) (\eta^\delta_{k}(s) - \eta^\delta_{k} \circ \chi(s)) \right| \\
\leq \frac{1}{2} \tilde{\Gamma}_n \mathbb{E} \left( \sum_{j=1}^{n} |\eta^\delta_j(s) - \eta^\delta_j \circ\chi(s)| \right)^{2}.
\end{multline*}
Given $\eta^\delta\circ\chi(s)$, the variables $| \eta^\delta_i(s) - \eta^\delta_i\circ\chi(s)|$ for $i=1,\dots,n$ are conditionally independent Bernoulli random variables with
$\mathbb{E} | \eta^\delta_i(s) - \eta^\delta_i\circ\chi(s) | \leq 
2\delta \|q\|_{\infty}$.
Therefore,
\begin{align*}
 \mathbb{E} \left( \sum_{j=1}^{n} |\eta^\delta_j(s) - \eta^\delta_j \circ\chi(s)| \right)^{2} & \leq 2n\delta \|q\|_{\infty} + 4(n\delta)^{2} \|q\|_{\infty}^{2},
\end{align*}
from which the result follows directly. 
\end{proof}
For $D_+^{(2)}$, we let $ \tilde{\eta}^\delta $ denote the martingale part of $\eta^\delta(t)$:
\[
\tilde{\eta}^\delta_{i}(t) = \eta^\delta_i(t) - \int^{t}_{0} (1-\eta^\delta_i(s)) q_i^+(\eta^\delta\circ\chi(s)) - \eta^\delta_i(s) q_i^-(\eta^\delta\circ\chi(s)) \mathd s.
\]
Now, $D_+^{(2)} \leq \mathcal{D}_1 + \mathcal{D}_2^+ + \mathcal{D}_2^- + \mathcal{D}_3$,
where
\begin{align*}
\mathcal{D}_1 & = \frac1n\mathbb{E}\int^{T}_{0} \sum_{i=1}^{n}  \left|\sum_{j=1}^{n} \partial_{j} q_i^+(\eta^\delta \circ \chi(s)) \frac{\tilde{\eta}^\delta_{j} (s) - \tilde{\eta}^\delta_{j} \circ \chi(s)}{\delta} \right| \mathd s  \\
\mathcal{D}_2^\pm & = \frac1n\mathbb{E}\int^{T}_{0} \sum_{i,j=1}^{n} \left| \frac{\partial_{j} q_i^+(\eta^\delta \circ \chi(s))}{\delta} \int^{s}_{\chi(s)} (\eta^\delta_j\circ\chi(u) - \eta^\delta_j(u) ) q^\pm_{j}(\eta^\delta\circ\chi(u)) \mathd u\right|  \mathd s  \\
\mathcal{D}_3 &= \mathbb{E}\sup_{t\in[0,T]} \sup_{\phi \in \Phi} \left|
\frac{1}{2} \int^{t}_{0}  \frac1n\sum_{i,j=1}^{n} \phi_{i} (1-\rho_{i}(s)) \partial_{j} q_i^+(\rho(s)) q_{j}(\rho(s)) \mathd s  \right.\\
&\qquad\left. -
\int^{t}_{0} \frac1n \sum_{i,j=1}^{n} \phi_{i} (1-\eta^\delta_i(s)) \partial_{j} q_i^+(\eta^\delta \circ \chi(s))  q_{j} (\eta^\delta\circ \chi(s)) \frac{s-\chi(s)}{\delta} \mathd s  \right|.
\end{align*}
are each to be bounded in turn, beginning with $\mathcal{D}_1$ in Lemma \ref{lem:D1LemmaBound}.
\begin{lemma}
\label{lem:D1LemmaBound}
For any $T > 0$, there is $
\mathcal{D}_1 \leq \frac{2T}{\sqrt{n\delta}}\|q\|_{\infty}^{1/2} (\frac{1}{\sqrt{n}}\|Dq^+\|_{2,1})$.
\end{lemma}
\begin{proof}
Letting $t_r = (r \delta) \wedge T$, first observe that
\[
\mathcal{D}_1
\leq (n\delta)^{-1} \sum_{r=0}^{\lfloor T / h\rfloor} (t_{r+1}-t_r) \sum_{i=1}^n
\mathbb{E}\sup_{s \in [t_r,t_{r+1}]}\left|
\sum_{j=1}^n \partial_j q_i^+(\eta^\delta(t_r))(\tilde{\eta}^\delta_j(s)-\tilde{\eta}^\delta_j(t_r))\right|.
\]
This quantity can be controlled in a similar way to $T_{i,1}$ of the proof of
Theorem \ref{thm:MidPointRate}. For each $r=0,\dots,\lfloor T/h\rfloor$, denote
the martingale
\[
M_i^{(r)}(t)=\sum_{j=1}^n \partial_j q_i^+(\eta^\delta(t_r))(\tilde{\eta}^\delta_j(t)-\tilde{\eta}^\delta_j(t_r)),
\qquad \mbox{for }t \geq t_r.
\]
Using Doob's inequality for martingales \cite[eq. (7.38)]{klebaner2012introduction},
\[
\mathbb{E}\sup_{s\in[t_r,t_{r+1}]} M_i^{(r)}(t)^2 \leq 4 [\mathbb{E}\langle M_i^{(r)}\rangle_{t_{r+1}} - \mathbb{E}\langle M_i^{(r)} \rangle_{t_r}],
\]
where $\langle M_i^{(r)}\rangle$ is the predictable quadratic variation of $M_i^{(r)}$, given by
\begin{multline*}
\langle M_i^{(r)} \rangle_t =  \int^{t}_{0} \sum_{j=1}^{n}  | \partial_{j} q_i^+(\eta^\delta(t_r)) |^{2} ( (1-\eta^\delta_j(s)) q^+_{j}(\eta^\delta\circ \chi(s)) \\+ \eta^\delta_j(s) q^-_{j}(\eta^\delta\circ\chi(s))) \mathd s. 
\end{multline*}
Therefore, $\mathbb{E}\sup_{s\in[t_r,t_{r+1}]} M_i^{(r)}(t)^2 \leq 4 \delta \|q\|_{\infty}\sum_{j=1}^n \|\partial_j q_i^+\|_{\infty}^2$,
which implies the lemma.
\end{proof}

\begin{lemma}
\label{lem:D2LemmaBound}
For any $ T > 0 $, $\mathcal{D}_2^+ + \mathcal{D}_2^- \leq 4 \delta T \|q\|_{\infty}^2 \|Dq^+\|_1$.
\end{lemma}

\begin{proof}
In terms of the number of jumps in $ \eta^\delta_i $ in the interval $ [a,b] $,
\begin{multline*}
\left|  \int^{s}_{\chi(s)} (\eta^\delta_j \circ \chi(u) - \eta^\delta_j(u)) q^+_{j}(\eta^\delta\circ\chi(u)) \mathd u \right| \\ \leq  \|q^+_{j} \|_{\infty}  \int^{s}_{\chi(s)} |\eta^\delta_j \circ \chi(u) - \eta^\delta_j(u) | \mathd u
\leq \delta  \|q^+_{j} \|_{\infty}  \mathbb{I}\left(\sum_{u\in[\chi(s),s]} \eta^\delta_j(u) \geq 1 \right).
\end{multline*}
Therefore, 
\begin{multline}
\mathcal{D}_2^+ \leq  \mathbb{E} \int^{T}_{0} n^{-1}  \sum_{i,j=1}^{n} \|q^+_{j} \|_{\infty} \|  \partial_{j} q_i^+\|_{\infty} \mathbb{I}\left(\sum_{u\in[\chi(s),s]} \eta^\delta_j(u) \geq 1 \right) \mathd s \\
\leq   \delta \|q\|_{\infty} \|Dq^+\|_{1} n^{-1} \sum_{j=1}^{n} \mathbb{E} \sum_{s\in[0,T]} \Delta \eta^\delta_j(s). \label{exact:lem1:eq3}
\end{multline}
It is clear from (\ref{model:eq5}) that $\mathbb{E}\sum_{s\in[0,T]}\Delta \eta^\delta_i(s) \leq 2 T \|q\|_{\infty}$, 
which may be substituted into (\ref{exact:lem1:eq3}). Applying the same arguments to $\mathcal{D}^-_2$ completes the proof.
\end{proof}

It now only remains to control $\mathcal{D}_3$. To do so, in Lemma \ref{lem:D3Part1}, we first bound the error incurred by estimating $\partial_j q_i^+(\eta^\delta \circ \xi(s))q_j(\eta^\delta \circ \chi(s))$ by \preprintlb $\partial_j q_i^+(\rho \circ \chi(s)) q_j(\rho \circ \chi(s))$. 

\begin{lemma}
\label{lem:D3Part1}
For any $ T > 0 $,
\begin{multline*}
\mathbb{E} \int^{T}_{0} \frac1n \sum_{i=1}^{n}  \left| \sum_{j=1}^{n} \partial_{j}q_i^+ (\eta^\delta\circ \chi(s)) q_{j}  (\eta^\delta\circ \chi(s)) -  \partial_{j}q_i^+ (\rho\circ\chi (s)) q_{j} (\rho\circ\chi (s))\right|  \mathd s  \\
\lesssim \frac{T^{2}}{n}e^{2T(\|q\| _{\infty}+\|D^{\ast}q\|_{1})}(n\|q\| _{\infty}\tilde{\Gamma}_{n}+\| D^{\ast}q\| _{1}^{2})(n\delta\| q\| _{\infty}\| D^{\ast}q\|_{1}+\| D^{\ast}q\| _{2,1}+\gamma_{n}) \\
+ n^{-1/2} T \|q\|_{\infty} (n\tilde{\Gamma}_n + n^{1/2}\theta_n) + n^{-1/2} T \|D^{\ast}q\|_F.
\end{multline*}
\end{lemma}
\begin{proof}
For the sake of brevity, let $ d_{ij}(x) = \partial_{j} q_i^+(x) q_{j}(x) $. 
It is a matter of direct computation to show that
\begin{equation}
\label{eq:DijSumFirstDerivs}
\max_{k=1,\dots,n}\sum_{i,j=1}^{n} \| \partial_{k} d_{ij}\|_{\infty} \leq n \|q \|_{\infty}\tilde{\Gamma}_n + \|Dq^+\|_{1} \|D^{\ast} q\|_{1}
\end{equation}
and
\begin{align}
\sum_{i=1}^{n} \left( \sum_{k=1}^{n} \left( \sum_{j=1}^{n} \| \partial_{k} d_{ij} \|_{\infty} \right)^{2} \right)^{1/2} 
&\lesssim n^{3/2} \|q\|_{\infty} \tilde{\Gamma}_n + \|(Dq^+)(D^{\ast}q)\|_{2,1} \nonumber \\
&\lesssim n^{3/2} \|q\|_{\infty} \tilde{\Gamma}_n + n^{1/2} \|D^{\ast}q\|_F^2,
\label{eq:DijFrobFirstDerivs}
\end{align}
where we once again make use of the fact that $\|A\|_{2,1} \leq n^{1/2} \|A\|_F$ for any $n\times n$ matrix $A$. For the second derivatives, there is
\begin{equation}
\label{eq:DijSumSecondDerivs}
\sum_{i,j,k=1}^{n}\| \partial_{k}^{2} d_{ij} \|_{\infty} \lesssim n \|q\|_{\infty} \theta_n + n \tilde{\Gamma}_n \| D^{\ast}q\|_{1} + \|Dq^+\|_{1} \gamma_n.
\end{equation}
Applying the Mean Value Theorem for $ d_{ij} $ together with Theorem \ref{thm:EulerRate} and (\ref{eq:DijSumFirstDerivs}) gives,
\begin{align}
\label{eq:DijMVT}
\MoveEqLeft\mathbb{E} \int^{T}_{0} n^{-1} \sum_{i=1}^{n}  \left| \sum_{j=1}^{n} d_{ij}(\eta^\delta\circ\chi(s)) - d_{ij}(\eta\circ\chi(s))\right| \mathd s \nonumber\\
&\leq \int^{T}_{0} n^{-1}\sum_{i,j,k=1}^{n} \| \partial_{k} d_{ij}\|_{\infty} \mathbb{E} | \eta^\delta_{k} \circ\chi(s) - \eta_{k}\circ\chi(s)| \mathd s \nonumber\\
&\lesssim \delta T^{2} (n \|q \|_{\infty}\tilde{\Gamma}_n + \|Dq^+\|_{1} \|D^{\ast} q\|_{1}) \|q\|_{\infty} \|D^{\ast} q\|_{1} e^{2T(\|q\|_{\infty} + \|D^{\ast} q\|_{1})}.
\end{align}
Finally, by applying Lemma \ref{lem:JWBarBoundCor} with $ f_{i}(x) = \sum_{j=1}^{n} d_{ij}(x) $ together with (\ref{eq:DijSumFirstDerivs}), (\ref{eq:DijFrobFirstDerivs}), and (\ref{eq:DijSumSecondDerivs}) implies
\begin{multline*}
\mathbb{E} \int^{T}_{0} n^{-1} \sum_{i=1}^{n} \left| \sum_{j=1}^{n} d_{ij}(\eta\circ\chi(s)) - d_{ij}(\rho\circ\chi(s))\right| \mathd s  \\
\lesssim n^{-1}  T^{2}  (n\|q\|_{\infty}\tilde{\Gamma}_n + \|Dq^+\|_1 \|D^{\ast}q\|_1) (\|D^{\ast}q\|_{2,1}+\gamma_n) e^{2T\|D^{\ast}q\|_{1}} \\
+ n^{-1/2} T \|q\|_{\infty} (n\tilde{\Gamma}_n + n^{1/2}\theta_n) + n^{-1/2} T \|D^{\ast}q\|_F,
\end{multline*}
which, together with (\ref{eq:DijMVT}), implies the lemma.
\end{proof}

Next, in Lemma \ref{lem:D3Part2}, we bound the error incurred by replacing the remaining occurrence of $\eta_i^\delta$ by $\rho_i$. For brevity, we let
\[
\mathcal{Q}_i^{\rho,\delta}(s) = \sum_{j=1}^n \partial_j q_i^+(\rho \circ \chi(s)) q_j(\rho \circ \chi(s)) \frac{s - \chi(s)}{\delta},\qquad s \geq 0,
\]
for each $i=1,\dots,n$.
\begin{lemma}
\label{lem:D3Part2}
For any $ T > 0 $,
\begin{multline*}
\mathbb{E} \sup_{t\in[0,T]} \sup_{\phi \in \Phi} \left|\int^{t}_{0} n^{-1} \sum_{i=1}^{n} \phi_{i} (\rho_{i}(s) -\eta^\delta_i(s)) \mathcal{Q}_i^{\rho,\delta}(s) \right| \\
\lesssim T \|Dq^+\|_{\infty} \|q\|_{\infty} (n^{-1/2} + T(\|D^{\ast}q\|_{2,1}+\gamma_n) \\+ \delta T \|q\|_{\infty}\|D^\ast q\|_1 ) e^{2T(\|q\|_{\infty} + \|D^{\ast} q\|_{1})}.
\end{multline*}
\end{lemma}

\begin{proof}
First, since $t - \chi(t) \leq \delta$ for any $t > 0$, Theorem \ref{thm:EulerRate} implies
\begin{multline}
 \mathbb{E} \sup_{t\in[0,T]} \sup_{\phi \in \Phi} \left|\int^{t}_{0} n^{-1} \sum_{i=1}^{n} \phi_{i} (\eta_{i}(s) -\eta^\delta_i(s)) \mathcal{Q}_i^{\rho,\delta}(s)  \mathd s \right| \\
 \leq 4\delta T^{2}  \|q\|^{2}_{\infty} \|Dq^+\|_{\infty} \|D^{\ast} q\|_{1} e^{2T(\|q\|_{\infty} + \|D^{\ast} q\|_{1})}. \label{eq:D3Part2First}
\end{multline}
Now, we let $\Psi(s;\boldsymbol{z})$ denote the set of vectors \[\Psi(s;\boldsymbol{z}) = \{\psi: \psi_{i} = \phi_{i} \mathcal{Q}_i^{\rho,\delta}(s), \; \phi \in \Phi \}.\]
Once again, using \cite[Lemma 26.9]{SSBD2014} and (\ref{eq:RadBound}), the Rademacher complexity of this set can be bounded by 
\begin{equation}
\Rad(\Psi(s;\boldsymbol{z})) \leq \max_{i} |\mathcal{Q}_i^{\rho,\delta}(s)| \Rad(\Phi(\boldsymbol{z}))
\lesssim n^{-1/2} \|q\|_{\infty} \|Dq^+\|_{\infty}. \label{eq:D3Part2Rad}
\end{equation}
Now, by using the symmetrisation method, comparing $\eta_i$ to $\omega_i$ and then $\omega_i$ to $\rho_i$, we find
\begin{multline}
\mathbb{E} \sup_{t\in[0,T]} \sup_{\phi \in \Phi} \left|\int^{t}_{0} n^{-1} \sum_{i=1}^{n} \phi_{i} (\eta_{i}(s) -\rho_{i}(s)) \mathcal{Q}_i^{\rho,\delta}(s) \mathd s \right|\\
\leq T \|Dq^+\|_{\infty} \|q\|_{\infty} \mathbb{E}\bar{J}(T) + \int^{T}_{0} \Rad (\Psi(s;\boldsymbol{z})) \mathd s. \label{eq:D3Part2Last}
\end{multline}
The lemma follows by combining (\ref{eq:D3Part2First}), (\ref{eq:D3Part2Rad}), and (\ref{eq:D3Part2Last}). 
\end{proof}

Finally, in Lemma \ref{lem:D3Part3}, we bound the remaining component of $\mathcal{D}_3$. 
\begin{lemma}
For any $T > 0$,
\label{lem:D3Part3}
\begin{multline*}
\sup_{t\in[0,T]} \sup_{\phi \in \Phi} \left|\frac{1}{2} \int^{t}_{0}  n^{-1}\sum_{i=1}^{n} \phi_{i} (1-\rho_{i}(s)) \sum_{j=1}^{n} \partial_{j} q_i^+(\rho(s)) q_{j}(\rho(s)) \mathd s \right. \\
\left.  - \int^{t}_{0} n^{-1} \sum_{i=1}^{n} \phi_{i} (1-\rho_{i}(s)) \mathcal{Q}_i^{\rho,\delta}(s) \mathd s  \right| \\
\lesssim \delta \|q\|_{\infty} (1+T)[n\tilde{\Gamma}_n \|q\|_{\infty} + \|Dq^+\|_{1} (1 + \|Dq\|_{1} + \|q\|_{\infty})].
\end{multline*}
\end{lemma}
\begin{proof}
One can show that for any real valued differentiable functions $ f $ and $ g $,
\begin{multline}
\label{eq:SecondOrderIntegralBound}
\left| \int^{T}_{0} g(s) \left[ (s-\chi(s))\delta^{-1} f \circ\chi(s) - \frac12 f(s) \right] \mathd s \right| \\ \leq 2 \delta T \left( \|f\|_{\infty} \|g^{\prime}\|_{\infty} + \|g\|_{\infty} \|f^{\prime}\|_{\infty} \right) + \delta \|g\|_{\infty} \|f\|_{\infty}.
\end{multline}
The result follows from applying (\ref{eq:SecondOrderIntegralBound}) with $ g(s) = \phi_{i} (1-\rho_{i}(s)) $ and $ f(s) = \partial_{j} q_i^+(\rho(s)) q_{j}(\rho(s)) $ for each $i,j=1,\dots,n$.
\end{proof}
Theorem \ref{thm:EulerExact} now follows from Gronwall's inequality with respect to (\ref{EA:main}) and (\ref{eq:SEDecomp}), with the additional terms bounded according to Lemmas \ref{lem:MartingaleExactBound}--\ref{lem:D3Part3} and equation (\ref{eq:Remainder2}).


\end{document}